\documentclass{amsart}
\usepackage[utf8]{inputenc}
\usepackage{amsmath,amsthm,amssymb}
\usepackage[all,cmtip]{xy}
\usepackage{color}
\usepackage{tikz}
\usetikzlibrary{matrix,arrows}
\newtheorem{theorem}{Theorem}[section]
\newtheorem{lemma}[theorem]{Lemma}
\newtheorem{proposition}[theorem]{Proposition}
\newtheorem{corollary}[theorem]{Corollary}
\newtheorem{claim}[theorem]{Claim}

\theoremstyle{definition}
\newtheorem{definition}[theorem]{Definition}

\newtheorem{problem}[theorem]{Problem}
\newtheorem{example}[theorem]{Example}

\newtheorem{definitions and remarks}[theorem]{Definitions and Remarks}

\theoremstyle{remark}

\newtheorem{remark}[theorem]{Remark}
\newtheorem{notation}[theorem]{Notation}

\numberwithin{equation}{section}
\newcommand{\umon}[1]{\pmb{u}^{\boldsymbol{#1}}}

\newcommand{\xmon}[1]{\pmb{x}^{\boldsymbol{#1}}}

\newcommand{\sm}[1]{{\scriptstyle (#1)}}

\begin{document}
\title[Local Resolution of Ideals Subordinated to a Foliation]{Local Resolution of Ideals Subordinated to a Foliation}

\author[A.~Belotto]{Andr\'e Belotto da Silva}
\address{University of Toronto, Department of Mathematics, 40 St. George Street,
Toronto, ON, Canada M5S 2E4}
\email[A.~Belotto]{andre.belottodasilva@utoronto.ca}


\keywords{Foliation, Resolution of singularities, Monomial}

\maketitle

\section*{Abstract}
Let $M$ be a complex- or real-analytic manifold, $\theta$ be a singular distribution and $\mathcal{I}$ a coherent ideal sheaf defined on $M$. We prove the existence of a local resolution of singularities of $\mathcal{I}$ that preserves the class of singularities of $\theta$, under the hypothesis that the considered class of singularities is invariant by $\theta$-admissible blowings-up. In particular, if $\theta$ is monomial, we prove the existence of a local resolution of singularities of $\mathcal{I}$ that preserves the monomiality of the singular distribution $\theta$.
\setcounter{tocdepth}{3}
%
\section{Introduction}
The subject of this article is resolution of singularities of an analytic ideal or variety which preserve the ``class'' of singularities of an ambient foliation. Let $M$ be a complex- or real-analytic manifold, $\mathcal{I}$ be a coherent and everywhere non-zero \emph{ideal sheaf}, $E$ be a \emph{simple normal crossing divisor} (i.e., an ordered collection $E = (E^{(1)},...,E^{(l)})$, where each $E^{(i)}$ is a smooth divisor on $M$ such that $\sum_i E^{(i)}$ is a reduced divisor with simple normal crossings) and $\theta$ be an \emph{involutive singular distribution} tangent to $E$ (i.e a coherent sub sheaf of the sheaf of vector fields over $M$ tangent to $E$, denoted by $Der_M(-logE)$, such that for each point $p$ in $M$, the stalk $\theta \cdot \mathcal{O}_p$ is closed under the Lie bracket operation). Note that $\theta$ generates a singular foliation over $M$ (by the Stefan-Sussmann Theorem \cite{Ste,Suss}). The triple $(M,\theta,E)$ will be called a \textit{foliated manifold} and the quadruple $(M,\theta, \mathcal{I},E)$ a \textit{foliated ideal sheaf}.
 
In here, a class of singularities (satisfying a property) $P$ stands for a collection of singular distributions whose singularities satisfy property $P$. For example, we say that $\theta$ is in the class of log-canonical singularities, if all singularities of $\theta$ are log-canonical (introduced by McQuillan in \cite{Mc}).

\begin{problem}
\label{prb:main}
To find a resolution of singularities of $\mathcal{I}$ (i.e. a composite of blowings-up $\sigma: \widetilde{M} \to M$ such that the pull-back $\widetilde{\mathcal{I}}$ of $\mathcal{I}$ is principal with support contained in a SNC divisor $\widetilde{E}$) such that the transform $\widetilde{\theta}$ of the ambient singular distribution $\theta$ (defined as the strict transform of $\theta$ intersected with $Der(-log\widetilde{E})$) is in the same class of singularities of $\theta$ (e.g. log-canonical, simple, monomial - we need to specify the class of singularities we are interested in). To simplify notation, we will denote the composition of blowings-up as $\sigma: (\widetilde{M},\widetilde{\theta}, \widetilde{\mathcal{I}},\widetilde{E}) \to (M,\theta, \mathcal{I},E)$.
\end{problem}
\begin{remark}
In general, given a class of singularities of $\theta$, blowings-up will not preserve this class of singularities. Indeed, over a three dimensional manifold, consider $\mathcal{I} = (y,z)$ and the log-canonical (i.e \emph{elementary} by \cite[Fact I.ii.4]{McP}) foliation $\theta = (\partial_x + x\partial_z)$. Then, the transform of $\theta$ by the blowing-up with center $V(y,z)$ (which principalize $\mathcal{I}$) is not log-canonical (see example \ref{ex:1} below).
\label{rk:ex1}
\end{remark}
In this manuscript we consider in details \emph{$\mathcal{K}$-monomial singular distributions}, defined in section \ref{ssec:Mon}, where $\mathcal{K}$ is a field between $\mathbb{Q}$ and the base field $\mathbb{K} =\mathbb{R}$ or $\mathbb{C}$. The foliation associated to it corresponds locally to the level curves of a monomial map (see \cite[Definition 1.1]{Cut1} for a definition of monomial maps), i.e. of a map whose output are monomials, with exponents in $\mathcal{K}$, in a special choice of coordinate system (see definition \ref{def:MSD} and Lemma \ref{lem:Fi}). This is possibly the smallest class of singularities where the problem \ref{prb:main} has a positive answer. Note that $\mathcal{K}$-monomial singular distributions are log-canonical singular distributions, moreover, and are useful in the study of first integrals and families of ideal sheaves and vector fields. In particular, any analytic singular foliation generated by a system of first integrals can be reduced, at least locally, to a $\mathbb{Q}$-monomial singular distribution (\cite{Bel2}, c.f \cite{Cut1}).

In \cite{BeloT,Bel1}, problem \ref{prb:main} is solved for $\mathcal{K}$-monomial and log-canonical singular distributions, under the additional hypothesis that $\theta$ has leaf dimension $1$ (i.e. when all leaves of the associated foliation have dimension smaller or equal to $1$) or $\mathcal{I}$ is invariant by $\theta$. In this manuscript, we prove that problem \ref{prb:main} can always be solved locally, i.e. we allow sequences of \emph{local blowings-up}, which is the composition of a blowing-up with an open immersion. More precisely, in the case of $\mathcal{K}$-monomial singular distributions, we prove the following:
\begin{theorem}[Main Theorem 1]\label{thm:main}
Let $(M,\theta,\mathcal{I},E)$ be a foliated ideal sheaf where $\theta$ is $\mathcal{K}$-monomial. Then, for every compact set $K$ in $M$, there exists a finite collection of morphisms $\tau_i : (M_i,\theta_i,\mathcal{I}_i,E_i) \rightarrow (M,\theta,\mathcal{I},E)$ such that:
\begin{itemize}
\item The pull-back of $\mathcal{I}$ is a principal ideal sheaf with support in $E_i$;
\item The singular distribution $\theta_i$ is $\mathcal{K}$-monomial for all $i$.
\item The morphism $\tau_i$ is a finite composition of admissible \emph{local} blowing-ups. 
\item There exists a compact set $K_i \subset M_i$ such that $\bigcup \tau_i(K_i)$ is a compact neighborhood of $K$.
\end{itemize}
\end{theorem}
\begin{remark}
See Theorem \ref{thm:main2} below for a version valid for more general classes of singularities. 

\end{remark}

One of the interests of this result is for reduction of singular foliations in higher dimensions (see \cite{Can,McP,Pan2} for results in reduction of singularities of foliations in dimension three). For example, the above theorem is used in \cite{Bel2} in order to prove that a foliation which is generated by a system of first integrals can be locally reduced to a monomial singular foliation (c.f. local monomialization of analytic morphisms via etoiles \cite{Cut1}). Theorem \ref{thm:main} may be useful, moreover, in the study of \emph{equiresolution} and \emph{resolution in families}, since a family structure generates a $\mathbb{Q}$-monomial foliation.
\subsection{Overview of the proof}
\label{ssec:Over}
The originality of this result comes from the fact that the desired resolution is not given by a standard resolution of singularities (e.g. \cite{BM,Vil4}). Let us consider in detail the example given in Remark $\ref{rk:ex1}$:
\begin{example}
\label{ex:1}
Consider a three dimensional regular variety $M = \mathbb{C}^3$, with globally defined coordinates $(x,y,z)$, and let $\mathcal{I}=(y,z)$ and $\theta = (\partial_x + x\partial_z)$, which is a regular singular distribution and therefore log-canonical (and monomial). The standard algorithms of resolution of singularities will demand us to blow-up with center $\mathcal{C}=V(y,z)$. As a consequence, in the $y$-chart $(x,y,z)=(\widetilde{x},\widetilde{y},\widetilde{y}\widetilde{z})$, the transform of $\theta$ is generated by $\widetilde{y}\partial_{\widetilde{x}} + \widetilde{x}\partial_{\widetilde{z}} $. Note that the linear part of this vector field is nilpotent and therefore $\theta$ is not log-canonical by \cite[Fact I.ii.4]{McP} (in particular, it is not monomial).\\
\\
In order to resolve the singularities of $\mathcal{I}$ and preserve the log-canonical (or monomial) class of singularities of $\theta$, we must first blow-up the origin. In the $x$-chart $(x,y,z)=(\widetilde{x},\widetilde{x}\widetilde{y},\widetilde{x}\widetilde{z})$, the transform of $\theta$ is generated by:
\[
\widetilde{x}\partial_{\widetilde{x}} -\widetilde{x} \widetilde{y}  \partial_{\widetilde{y}} + \widetilde{x}(1 - \widetilde{z} ) \partial_{\widetilde{z}}
\]
which is log-canonical (the linear part in not nilpotent). We must blow-up once again the origin of this chart, before blowing-up the regular curve given by $V(\widetilde{y},\widetilde{z})$, so to guarantee that the transforms of $\theta$ have only log-canonical singularities through the process (i.e. in order to preserve the class of singularities of $\theta$).
\end{example}
The example shows that we should impose some restriction to the centers of blowing up. In section \ref{sec:tadm}, we give the definition a notion of $\theta$\textit{-admissible} blowings-up (see definition \ref{def:tadm}) which is first introduced in \cite{BeloT} (c.f the one leaf-dimensional case \cite[Section 2.3]{Bel1}). We show that $\theta$-admissible blowings-up preserve the class of $\mathcal{K}$-monomial singularities (Proposition \ref{prop:AdmCenter}). As a consequence, we only need to find a resolution of singularities of $\mathcal{I}$ by $\theta$-admissible blowings-up. Indeed, we will prove the following more general result:
\begin{theorem}[Main Theorem 2]\label{thm:main2}
Let $(M,\theta,\mathcal{I},E)$ be a foliated ideal sheaf where the singularities of $\theta$ are of certain class which are also log-canonical. Suppose that this class of singularities is invariant by $\theta$-admissible blowings-up, i.e. an analogous of Proposition \ref{prop:AdmCenter} is valid. Then, for every compact set $K$ in $M$, there exists a finite collection of morphisms $\tau_i : (M_i,\theta_i,\mathcal{I}_i,E_i) \rightarrow (M,\theta,\mathcal{I},E)$ such that:
\begin{itemize}
\item The pull-back of $\mathcal{I}$ is a principal ideal sheaf with support in $E_i$;
\item The singular distribution $\theta_i$ is in the same class of singularities as $\theta$.
\item The morphism $\tau_i$ is a finite composition of $\theta$-admissible \emph{local} blowing-ups (in particular, all intermediate transforms of $\theta$ are in the same class of singularities as $\theta$). 
\item There exists a compact set $K_i \subset M_i$ such that $\bigcup \tau_i(K_i)$ is a compact neighborhood of $K$.
\end{itemize}
\end{theorem}
\begin{remark}
In what follows, we prove the theorem in the case that the compact set $K$ is a point $p \in M$. The result for more general compact sets easily follows.
\end{remark}
Standard resolution of singularities algorithms of $\mathcal{I}$ are based on decreasing the order of $\mathcal{I}$ (see \cite{Ko} and references therein). Unfortunately, the order of $\mathcal{I}$ has no relation with the singular distribution and, therefore, is not fit for problem \ref{prb:main}. In section \ref{sec:CIS}, we give the definition of an invariant $\nu$ (first introduced in \cite{BeloT,Bel1}), which we call \emph{tg-order} (see definition \ref{def:inv}). This invariant measures the order of tangency between an ideal sheaf $\mathcal{I}$ and a singular distribution $\theta$, even if the objects are singular.

The proof of Theorem \ref{thm:main2} follows by induction on the tg-order $\nu$. The induction has three main steps (as in \cite{AB} and inspired by \cite{Cut2}) which are presented in details in section \ref{sec:over} and proved in sections \ref{sec:type}, \ref{sec:prep} and \ref{sec:Drop}. Technically, the main difficulty is that the tg-order has no \emph{surface of maximal contact} associated to it, i.e. we can not use the standard ideas of Hironaka in order to argue by induction on the dimension of $M$. This comes as no surprise since we deal with singular foliations, where surfaces of maximal contact are rare. The same technical difficulty appears in the case when $\theta$ has leaf dimension one (i.e. all leaves have dimension smaller or equal to one), where the Problem \ref{prb:main} admits a global solution by \cite{Bel1}.

Nevertheless, the problem in general is more delicate since the invariant $\nu$ may \emph{increase} after $\theta$-admissible blowings-up (see example \ref{ex:main}). If the leaf dimension of $\theta$ is one the invariant $\nu$ does not increase since the transform of $\theta$ under $\theta$-admissible blowings-up follows a well-behaved dichotomy (see \cite[section 2.3]{Bel1}). In \cite{Bel1}, the proof is divided in two steps which explore this dichotomy in order to control $\nu$. In the case of a higher leaf-dimension, the transform of $\theta$ does not follow this dichotomy (see example \ref{ex:Kinds}(3)) and the invariant $\nu$ may increase.

In order to deal with this extra difficulty we blow-up centers which are \emph{not} necessarily contained in the support of $\mathcal{I}$ (as in \cite{Cut2}). This allow us to emulate the existence of a surface of maximal contact and to control the invariant $\nu$. The price is that we need to choose an special direction, which means that the algorithm is only local instead of global. Nevertheless, it is worth remarking that our algorithm can be globalized when the dimension of $M$ is smaller or equal to three (c.f. \cite{AB,Cut2}).

\begin{example}
We assume that the reader is familiar with definition \ref{def:inv} (of the pair of invariant $(\nu,type)$). Let $n \in \mathbb{N}$ be bigger then $2$ and consider: the three dimensional manifold $\mathbb{C}^3$ with coordinates $(x,y,z)$, the ideal sheaf $\mathcal{I} =(y^2 +x z^n+x^{n+1})$ and the involutive singular distribution $\theta = (\partial_y,\partial_z)$. Computing the tangency sequence of $(\theta,\mathcal{I})$ give us:
\[
H(\theta,\mathcal{I},1) = (y,x  z^{n-1},x^{n+1}), \quad H(\theta,\mathcal{I},2) = \mathcal{O}_M
\] 
which implies that: $type=1$ everywhere; $\nu\leq 2$ in every point in $M$; $\nu =2$ over the variety $\mathcal{C} = V(x,y)$. Let $\sigma : (\widetilde{M},\widetilde{\theta},\widetilde{\mathcal{I}},\widetilde{E}) \to (M,\theta,\mathcal{I},E)$ be the blowing-up with center $\mathcal{C}$. Note that this is a $\theta$-admissible blowing-up (see Definition \ref{def:tadm} and compare with example \ref{ex:Kinds}(3)). Consider the $x$-chart $(x,y,z) = (\tilde{x},\tilde{x}\tilde{y},\tilde{z})$, where $\widetilde{E}$ is locally generated by $(\widetilde{x}=0)$. Then:
\[
\widetilde{\mathcal{I}} = \tilde{x} \, (\tilde{y}^2\tilde{x} + \tilde{z}^n + \tilde{x}^n), \quad \widetilde{\theta} = (\partial_{\widetilde{y}}, \partial_{\widetilde{z}})
\]
which implies that the invariant $\nu$ has increased since it is equal to $n$ along the variety $V(\widetilde{x},\widetilde{z})$.

Let us briefly indicate how our algorithm would treat this example. We first blow-up with center $V(x,z)$ (instead of $V(x,y)$) in order to obtain, what we call, a \emph{prepared normal form} (see Proposition \ref{prop:prepnormal}). For instance, in the $x$-chart $(x,y,z) = (\tilde{x},\tilde{y},\tilde{x}\tilde{z})$ (where $\widetilde{E} = (\tilde{x}=0)$) we obtain:
\[
\widetilde{\mathcal{I}} = (\tilde{y}^2 + \tilde{x}^{n+1}(1+\tilde{z}^n)), \quad \widetilde{\theta} = (\partial_{\widetilde{y}}, \partial_{\widetilde{z}})
\]
Note that the center $V(x,z)$ leaves $y$ invariant and, therefore, the tangency order $\nu$ does not increase. Nevertheless, this center does not need to be global. Once in prepared normal form, we are able to control the invariant $\nu$ by sequences of blowings-up (see Proposition \ref{prop:Dropping}). In this example, we take a sequence of blowings-up which principalize $(\tilde{y}^2, \tilde{x}^{n+1})$ in order to decrease the invariant $\nu$.
\label{ex:main}
\end{example}
\section{Notation and background}
We consider a foliated manifold $(M,\theta,E)$ and a foliated ideal sheaf $(M,\theta, \mathcal{I},E)$. 
\subsection{Singular distributions} Let $Der_M$ denote the sheaf of analytic vector fields on $M$, i.e. the sheaf of analytic sections of $TM$. An {\em involutive singular distribution} is a coherent sub sheaf $\theta$ of $Der_M$ such that for each point $p$ in $M$, the stalk $\theta_p:=\theta \cdot \mathcal{O}_p$ is closed under the Lie bracket operation. Consider the quotient sheaf $Q = Der_M/ \theta$.  The {\em singular set} of $\theta$ is defined by the closed analytic subset $S(\theta) = \{p \in M : Q_p \text{ is not a free $\mathcal{O}_p$ module}\}$. A singular distribution $\theta$ is called \textit{regular} if $S(\theta)=\emptyset$. On $M \setminus S(\theta)$ there exists an unique analytic subbundle $L$ of $TM |_{ M\setminus S(\theta)}$ such that $\theta$ is the sheaf of analytic sections of $L$. We assume that the dimension of the $\mathbb{K}$ vector space $L_p$ is the same for all points $p$ in $M \setminus S$ (this always holds if $M$ is connected). It is called the {\em leaf dimension} of $\theta$ and denoted by $d$. In this case $\theta$ is called an involutive \textit{$d$-singular distribution}.

Let $Der_M(-logE)$ denote the coherent sub sheaf of $Der_M$ which contains all derivations tangent to $E$, i.e. given a point $p$, a derivation $\partial$ is in $Der_M(-logE)\cdot \mathcal{O}_p$ if and only if $\partial[\mathcal{I}_{E}] \subset \mathcal{I}_{E}$ (where $\mathcal{I}_E$ is the reduced principal ideal whose support is $E$). A singular distribution $\theta$ which is also a sub sheaf of $Der_M(-logE)$ is said to be tangent to $E$.
\subsection{Blowings-up and local resolution of singularities} We follow \cite{Ko}. A blowing-up $\sigma: \widetilde{M} \to M$ is said to be \emph{admissible} if the center of blowing-up $\mathcal{C}$ has normal crossings with $E$. In this case, we denote the blowing-up by $\sigma : (\widetilde{M},\widetilde{E}) \to (M,E)$, where $\widetilde{E}$ is the union of the inverse-image of $E$ with the exceptional divisor $F$ of the blowing-up. We denote by $\mathcal{I}_{F}$ the reduced ideal sheaf whose support is the divisor $F$.

Given an ideal sheaf $\mathcal{I}\subset \mathcal{O}_M$, we define the \emph{total transform} of $\mathcal{I}$ as the ideal sheaf $\mathcal{I}^{\ast}:=\sigma^{\ast}\mathcal{I}$. If the center of blowing-up is contained in the support of $\mathcal{I}$, we say that the blowing-up is of \emph{order $\geq1$}. In this case, the \emph{birational transform} $\mathcal{I}^c$ of the ideal sheaf $\mathcal{I}$ is the ideal sheaf that satisfies the equality $\mathcal{I}^{\ast} = \mathcal{I}_{F} \cdot \mathcal{I}^c$, which we will also denote by $\mathcal{I}^{c}=\mathcal{I}_{F}^{-1} \cdot \mathcal{I}^{\ast}$ (where $\mathcal{I}_{F}^{-1} $ stands for the $\mathcal{O}_M$ sub-sheaf of the meromorphic functions whose poles vanish over $F$). In other words, $\mathcal{I}^c$ is locally generated by functions of the form $x^{-1} f^{\ast}$, where $f \in \mathcal{I}$ and $(x=0)$ is a local equation of the divisor $F$.
\begin{remark}
The definition of birational transform we use corresponds to the birational transform of the marked ideal $(\mathcal{I},1)$ (see \cite[Def 3.60]{Ko}).
\end{remark}
We can now write an admissible blowing-up as $
\sigma: (\widetilde{M}, \widetilde{\mathcal{I}},\widetilde{E}) \to (M,\mathcal{I},E)$ where: either $\sigma$ is of order $\geq1$ and the ideal sheaf $\widetilde{\mathcal{I}}$ is the birational transform $\mathcal{I}^c$; or the ideal sheaf $\widetilde{\mathcal{I}}$ is the total transform $\mathcal{I}$. We need this dichotomy since we blow up centers outside of the support of $\mathcal{I}$ (see subsection \ref{ssec:Over}).

An admissible \textit{local} blowing-up $\tau: (\widetilde{M},\widetilde{\mathcal{I}},\widetilde{E}) \to (M,\mathcal{I},E)$ is the composition of an admissible blowing with an open immersion (e.g. a chart of the blowing-up). A \emph{sequence of local blowings-up} is a sequence of morphisms
\[
 \begin{tikzpicture}
  \matrix (m) [matrix of math nodes,row sep=3em,column sep=3em,minimum width=1em]
  {(M_r,\mathcal{I}_r,E_r) & \cdots & (M_0,\mathcal{I}_0,E_0)\\};
  \path[-stealth]
    (m-1-1) edge node [above] {$\tau_r$} (m-1-2)
    (m-1-2) edge node [above] {$\tau_1$} (m-1-3);
\end{tikzpicture}
\]
where each morphism is an admissible local blowing-up and the transforms are defined in the usual way (see \cite[Def 3.66]{Ko}). Finally, a local resolution of $\mathcal{I}$ at a point $p$ of $M$, is a finite collection of morphisms $\tau_i : (M_i,\mathcal{I}_i,E_i) \rightarrow (M,\mathcal{I},E)$ such that:
\begin{enumerate}
\item The ideal sheaf $\mathcal{I}_i = \mathcal{O}_{M_i}$, i.e. $\mathcal{I}_i$ is everywhere locally generated by units.
 \item The morphism $\tau_i$ is a finite composition of admissible local blowing-ups. 
 \item There exists a compact set $K_i \subset M_i$ such that $\bigcup \tau_i(K_i)$ is a compact neighborhood of $p$.
\end{enumerate}

\begin{remark}
In our Theorems $\ref{thm:main}$ and $\ref{thm:main2}$ we prove the existence of a local resolution of $\mathcal{I}$ at a point $p$ of $M$ (and therefore over any compact of $M$). Note that the total transform $\tau_i^{\ast}(\mathcal{I})$ is a principal ideal sheaf with support in $E_i$. Indeed, the total transform $\tau_i^{\ast}(\mathcal{I})$ is given by the product of ideal sheaves $\mathcal{I}_{F}$ and their pull-backs, i.e. it is locally generated by a monomial with support in the exceptional divisor $E_i$.
\end{remark}

\subsection{Blowing-up of a singular distributions}
Given a singular distribution $\theta\subset Der_M(-logE)$ and an admissible blowing-up $\sigma : (\widetilde{M},\widetilde{E}) \to (M,E)$, we denote by $\widetilde{\theta}$ the intersection of the strict transform of $\theta$ with $Der_{\widetilde{M}}(-log\widetilde{E})$. In particular, this guarantees that $(\widetilde{M},\widetilde{\theta},\widetilde{E})$ is a foliated ideal sheaf and we can write $\sigma : (\widetilde{M},\widetilde{\theta},\widetilde{E}) \to (M,\theta,E)$.
\section{Monomial Distributions}
\label{ssec:Mon}
\begin{definition}[Monomial singular distribution]
\label{def:MSD}
Given a foliated manifold $(M,\theta,\allowbreak E)$ and a field $\mathcal{K}$ such that $\mathbb{Q} \subset \mathcal{K} \subset \mathbb{K}$, we say that the singular distribution $\theta$ is $\mathcal{K}$-\textit{monomial} at a point $p$ if there exists set of generators $\{\partial_1,...,\partial_d\}$ of $\theta \cdot \mathcal{O}_p$ and a coordinate system $(\pmb{u},\pmb{w}) = (u_1,\ldots,  u_k,w_{k+1}, \ldots w_m)$ centered at $p$ such that:
\begin{itemize}
\item[(i)] Locally $E = \{u_1 \cdots u_l=0\}$, for some $l\leq k$;
\item[(ii)] $\theta$ is everywhere tangent to $E$, i.e. $\theta \subset Der_M(-logE)$;
\item[(iii)] The vector fields $\partial_i$ are of the form:
\[
 \begin{aligned}
 \partial_i &=  \partial w_{m+1-i} \text{, }i = 1, \ldots, m-k &\text{, and}\\
  \partial_i &= \sum^{k}_{j=1} \alpha_{i,j}u_j \partial u_j \text{, }i=m-k+1, \ldots, d&
 \end{aligned}
\]
where $\alpha_{i,j} \in \mathcal{K}$.
\item[(iv)] If $\omega \subset Der_M(-logE)$ is an involutive $d$-singular distribution such that $\theta \subset \omega$, then $\theta = \omega$.
\end{itemize}
In this case, we say that $(\pmb{u},\pmb{w})$ is a \textit{monomial coordinate system} and that $\{\partial_1,...,\partial_d\}$ is a \textit{monomial basis} of $\theta \cdot \mathcal{O}_p$.
\end{definition}
\begin{remark}[Geometrical Interpretation of $(iv)$]
Assuming conditions $[i-iii]$ above, Property $[iv]$ implies that the singularity set of $\theta$ is of codimension at least two outside of the exceptional divisor $E$.
\label{rk:Condiv}
\end{remark}
\begin{notation}[Monomial coordinate system]
We sometimes need to distinguish one of the non-exceptional coordinates $\pmb{w}$. To that end, we will denote by $(\pmb{u},v,\pmb{w})$ a monomial coordinate system where the vector field $\partial_{v}$ will always be assumed to be contained in $\theta\cdot \mathcal{O}_p$.
\end{notation}
The lemma below shows an important feature of $\mathcal{K}$-monomial singular distributions:
\begin{lemma}(Monomial First Integrals - \cite[Lemma 2.2.2]{BeloT})
Given a foliated manifold $(M,\theta,E)$, the singular distribution $\theta$ is $\mathcal{K}$-monomial if and only if for all points $p \in M$, there exists a monomial coordinate system $(\pmb{u},\pmb{w}) = (u_1,\ldots,  u_k,w_{k+1}, \ldots,\allowbreak w_m)$ centered at $p$ and $m-d$ monomials $(\umon{\beta_1},\dots,\umon{\beta_{m-d}})$, where $\umon{\beta_j} = \Pi_{i=1}^k u_i^{\beta_{i,j}}$ with $\beta_{i,j} \in \mathcal{K}$, such that
\begin{itemize}
\item the multi-indexes $\{\pmb{\beta}_1, \ldots, \pmb{\beta}_{m-d}\}$ span a $m-d$ subspace of $\mathcal{K}^{k}$,
\item $
\theta\cdot \mathcal{O}_p =\{\partial \in Der_p(-logE);\text{ } \partial(\umon{\beta_i})\equiv 0 \text{ for all } i\}
$.
\end{itemize}
\label{lem:Fi}
\end{lemma}
\begin{remark}
The lemma gives a precise relation between monomial singular distributions and monomial maps (see definition in \cite{Cut2}). If a holomorphic map is monomial, the foliation generated by its level curves is $\mathbb{Q}$-monomial. Furthermore, in the study of families, the notion of quasi-smoothness (see a definition in \cite{Vil2}) is closely related to a $\mathbb{Q}$-monomial distributions. We won't explicitly use Lemma \ref{lem:Fi} in this manuscript, so we omit its proof.
\end{remark}
We now turn to two properties of $\mathcal{K}$-monomial singular distributions which will be needed:
\begin{lemma}
If $\theta$ is $\mathcal{K}$-monomial singular distribution at a point $p$ in $M$, then there exists an open neighborhood $U$ of $p$ such that $\theta$ is $\mathcal{K}$-monomial at every point $q$ in $U$. Moreover, if $(\pmb{u},\pmb{w})$ is a coordinate system defined in a connected open neighborhood $V$ which is monomial at $p$, then $\theta$ is $\mathcal{K}$-monomial everywhere in $V$.
\label{lem:Gmonloc}
\end{lemma}
\begin{proof}
Conditions $[i]$, $[ii]$ and $[iv]$ are clearly open. In order to prove that condition $[iii]$ also is, fix a monomial coordinate system $(\pmb{u},\pmb{w})$ centered at $p$ and a monomial basis $\{\partial_1, \ldots, \partial_d\}$, both defined in an open neighborhood $U$ of $p$, and let $q$ be another point in $U$. Note that a translation in the non-exceptional variables $\pmb{w}$ preserves the $\mathcal{K}$-monomiality. Thus, without loss of generality, we can suppose that $q = (\pmb{\xi},\pmb{0}) = (\xi_1,\ldots, \xi_t,0, \ldots, 0)$, where $t\leq k$ and all terms $\xi_i \neq 0$. Moreover, we can suppose that all vector fields $\partial_i$ in the monomial basis are singular at $p$, i.e. $\partial_i =\sum^{k}_{j=1} \alpha_{i,j}u_j \partial u_j $. Let $s$ be the rank of the $d \times t$ matrix of coefficients $\pmb{A} = [\alpha_{i,j}]$. Without loss of generality, we can assume that:
\[
\pmb{A} = \left[ \begin{array}{cc}
\pmb{A}_1 & \pmb{A}_2 \\
0 & 0
\end{array}
 \right]
\] 
where $\pmb{A}_1$ is a $s \times s$-diagonal matrix. In particular
\[
\begin{aligned}
\partial_i &= \alpha_{i,i}u_i \partial_{u_i} +  \sum^{k}_{j=s+1} \alpha_{i,j} u_j \partial_{u_j} \text{, } i \leq s\\
\partial_i &= \phantom{\alpha_{i,i}u_i \partial_{u_i} + } \sum^{k}_{j=t+1} \alpha_{i,j} u_j \partial_{x_j} \text{, } i > s
\end{aligned}
\]
where $\alpha_{i,i}\neq 0$ for all $i\leq s$. Now, consider the coordinate system $(\pmb{x},\pmb{w})$ given by the following formulas:
\[
\begin{aligned}
u_j &= \xi_j exp(\alpha_{j,j}x_j) \text{, }j\leq s\\
u_j &= \xi_j exp\left( x_j + \sum_{i=1}^s \alpha_{i,j}x_i \right) \text{, }s < j\leq t\\
u_j &= x_j exp\left(\sum_{i=1}^s \alpha_{i,j}x_i \right) \text{, } t<j \leq k
\end{aligned}
\]
It is not difficult to see that $(\pmb{x},\pmb{w})$ is a coordinate system centered at $q$ and:
\[
\begin{aligned}
\partial_i &= \partial_{x_i} \text{, } i \leq s\\
\partial_i &=  \sum^{k}_{j=t+1} \alpha_{i,j} x_j \partial_{x_j} \text{, } i > s
\end{aligned}
\]
which finishes the proof.
\end{proof}
\begin{lemma}
Suppose that $\theta$ is a $\mathcal{K}$-monomial $1$-singular distribution and that the ideal sheaf $\mathcal{I}$ is $\theta$-invariant, i.e., $\theta[\mathcal{I}]\subset \mathcal{I}$. Given a point $p$ in $M$, fix a monomial basis $\partial$, a monomial coordinate system $(\pmb{u},\pmb{w})$ and a system of generators $(f_1,\ldots,f_n)$ of $\mathcal{I}$ at the point $p$. Then, there exists a system of generators $(h_1,\dots,h_N)$ of $\mathcal{I}\cdot \mathcal{O}_p$ such that $\partial h_i = \zeta_i h_i$, for constants $\zeta_i \in \mathcal{K}$. If $\partial$ is a regular vector field, moreover, then $\zeta_i =0$. Finally, each $h_i$ is a sum of blocks in the Taylor expansion of one of the $f_j$ (in relation to the fixed monomial coordinate system).
\label{lem:InvNSing}
\end{lemma}
\begin{proof}
Fix a monomial coordinate system $(\pmb{u},\pmb{w})$ and a system of generators $(f_1, \dots, \allowbreak f_n)$ of $\mathcal{I}\cdot \mathcal{O}_p$. We first assume that $\partial$ is a singular vector field. In this case we will also denote the monomial coordinate system $(\pmb{u},\pmb{w})$ by $\pmb{x} = (x_1, \ldots, x_m)$ and, therefore, $\partial = \sum \alpha_i x_i \partial_{x_i}$. So, given any monomial $\xmon{\beta}= \Pi_{i=1}^m x_i^{\beta_i}$ with $\beta_i \in \mathbb{N}$, we have that
\[
\partial(\xmon{\beta}) = \sum^m_{i=1} \alpha_i \beta_i \xmon{\beta} = \zeta_{\boldsymbol{\beta}} \xmon{\beta}
\]
For some $\zeta_{\boldsymbol{\beta}} \in \mathcal{K}$. Since the number of monomials in a Taylor expansion is countable, there exists a countable set $\widetilde{\mathcal{K}} \subset \mathcal{K}$ such that $\zeta_{\boldsymbol{\beta}} \in \widetilde{\mathcal{K}}$ for all $\boldsymbol{\beta} \in \mathbb{N}^{m}$. This allow us to rewrite the Taylor expansion of each generator $f_i$:
\[
f_i(\pmb{x})= \sum_{j \in \mathbb{N}} h_{i,j}(\pmb{x})
\]
where $\partial(h_{i,j})= \zeta_j h_{i,j}$ with $\zeta_j \in \widetilde{\mathcal{K}}$ and $\zeta_j \neq \zeta_k$ whenever $j\neq k$. Furthermore, let us note that $h_{i,j}(\pmb{x})$ are convergent Taylor series (because $f_i$ is absolutely convergent in a neighborhood), which implies that $h_{i,j}(\pmb{x}) \in \mathcal{O}_p$. We claim that all functions $h_{i,j}$ are contained in $\mathcal{I}\cdot \mathcal{O}_p$. Indeed, let us show that $h_{1,0}$ is in the ideal (the proof for the other functions is analogous). Let $g_0 = f_1$ and
\[
\begin{aligned}
 g_1 &= \frac{1}{\zeta_0-\zeta_1}(\partial(f_1) - \zeta_1 f_1 ) = \frac{1}{\zeta_0-\zeta_1}\left[ \sum_{j \in \mathbb{N}} \zeta_j h_{1,j}(x) - \zeta_1 \sum_{j \in \mathbb{N}} h_{1,j}(x) \right] \\
 &= h_{1,0} + \sum_{j \geq 2} \gamma_{1,j}\sm{1} h_{1,j}
\end{aligned}
\]
where $\gamma_{1,j}\sm{1} = \frac{\zeta_j-\zeta_1}{\zeta_0-\zeta_1}$. We recursively define $g_n \in \mathcal{I}\cdot \mathcal{O}_p$:
\[
g_n = \frac{1}{\zeta_0-\zeta_n}(\partial(g_{n-1}) - \zeta_n g_{n-1}) = h_{1,0} + \sum_{j \geq n+1} \gamma_{1,j}\sm{n} h_{1,j} 
\]
for constants $\gamma_{1,j}\sm{n}$. It is clear that $(g_n)_{n\in\mathbb{N}} \subset \mathcal{I}\cdot \mathcal{O}_p$ converges formally to $h_{1,0}(x)$, i.e in the Krull topology of $\widehat{\mathcal{O}}_p$. By faithful flatness, this implies that $h_{1,0} \in \mathcal{I}\cdot \mathcal{O}_p$ (c.f \cite{Hor} section 6.3). It is now clear that $\mathcal{I}\cdot \mathcal{O}_p = (h_{i,j})$ and we just need to use Noetherianity to obtain a finite system of generators.

If $\partial$ is a regular vector field we will also denote the monomial coordinate system $(\pmb{u},\pmb{w})$ by $(\pmb{x},y) = (x_1, \ldots, x_{m-1},y)$, where we assume that $\partial  = \partial_y$. Now, we can write the Taylor expansion of each generator $f_i$ as:
\[
f_i(\pmb{x},y)= \sum_{j \in \mathbb{N}} h_{i,j}(\pmb{x})y^j
\]
We claim that all functions $h_{i,j}$ are contained in $\mathcal{I}\cdot \mathcal{O}_p$. Indeed, this is analogous to the previous case. Finally, it is clear that $\zeta_i=0$ for every $i$ and we are done.
\end{proof}
\begin{corollary}
Suppose that $\theta$ is a $\mathcal{K}$-monomial singular distribution and that the ideal sheaf $\mathcal{I}$ is $\theta$-invariant, i.e., $\theta[\mathcal{I}]\subset \mathcal{I}$. Fix a monomial basis $\{\partial_1, \ldots, \partial_d\}$ at a point $p$. Then, there exists a system of generators $(f_1,\dots,f_n)$ of $\mathcal{I}\cdot \mathcal{O}_p$ such that $\partial_i f_j = \zeta_{i,j} f_j$, where $\zeta_{i,j} \in \mathcal{K}$. Moreover, if $\partial_i$ is a regular vector field, then $\zeta_{i,j} =0$.
\label{cor:InvaBase}
\end{corollary}
\begin{proof}
This result follows by induction on the leaf-dimension $d$. The case $d=1$ is given by Lemma \ref{lem:InvNSing}. The induction hypothesis means that the Corollary is valid for monomial singular distributions of leaf dimension $d-1$. Since $\{\partial_1, \ldots, \partial_{d-1}\}$ locally generates a $d-1$ monomial singular distribution, there exists a systems of generators  $(f_1,\dots,f_n)$ of $\mathcal{I}\cdot \mathcal{O}_p$ such that $\partial_i f_j = \widetilde{\zeta}_{i,j} f_j$, where $\widetilde{\zeta}_{i,j}$ in $\mathcal{K}$ and $i\leq d-1$. We apply Lemma \ref{lem:InvNSing} again to get a system of generators $(h_1,\ldots, h_N)$ such that $\partial_d h_j = \zeta_{d,j} h_j$. Finally, since $h_j$ is a sum of blocks in the Taylor expansion of one of the $f_i$, we conclude that $\partial_i h_j = \zeta_{i,j} h_j$ for constants $\zeta_{i,j} \in \mathcal{K}$ and $i\leq d$.
\end{proof}
\section{$\theta$-Admissible blowings-up}
\label{sec:tadm}
In this section we define the notion of $\theta$-admissible centers, which was first introduced in \cite{BeloT} (an specialization of the definition when $\theta$ has leaf dimension one can be found in \cite[section 2.3]{Bel1}). Given an ideal sheaf $\mathcal{I}$, we consider ideal sheaves $\Gamma_{\theta,k}(\mathcal{I})$, which we call \textit{generalized $k$-Fitting ideal}, whose stalk at each point $p$ in $M$ is generated by all terms of the form:
\[
det \left\| 
\begin{array}{ccc}
\partial _1(f_1) & \ldots & \partial_{1}(f_k) \\
 \vdots & \ddots & \vdots \\
 \partial_{k}(f_1) & \ldots & \partial_{k}(f_k) 
\end{array}
\right\|
\]
for $\partial_i \in \theta \cdot \mathcal{O}_{p}$ and $f_j \in \mathcal{I} \cdot\mathcal{O}_{p}$.
\begin{definition}[$\theta$-Admissible blowing-up] We say that an admissible blowing-up $\sigma:(\widetilde{M},\widetilde{\theta},\widetilde{E}) \to (M,\theta,E)$ is $\theta$\textit{-admissible} if there exists $k_0 \in \mathbb{N}$ (possibly $k_0=0$) such that:
\begin{enumerate}
\item The ideal sheaf $\Gamma_{\theta,k}(\mathcal{I}_{\mathcal{C}}) + \mathcal{I}_{\mathcal{C}}$ is equal to $\mathcal{O}_M$ for $k\leq k_0$;
\item The ideal sheaf $\Gamma_{\theta,k}(\mathcal{I}_{\mathcal{C}})+ \mathcal{I}_{\mathcal{C}}$ is contained in $\mathcal{I}_{\mathcal{C}}$ for $k>k_0$. 
\end{enumerate} 
where $\mathcal{I}_{\mathcal{C}}$ is the reduced ideal sheaf whose support is the center of the blowing-up.
\label{def:tadm}
\end{definition}

\begin{example}
\label{ex:Kinds}
We present four examples:
\begin{enumerate}
\item If the center $\mathcal{C}$ is $\theta$-invariant center (i.e if all leaves of $\theta$ that intersects $\mathcal{C}$ are contained in $\mathcal{C}$), the blowing-up is $\theta$-admissible.
\item If the center $\mathcal{C}$ is an admissible $\theta$-totally transverse (i.e all vector fields in $\theta$ are transverse to $\mathcal{C}$), the blowing-up is $\theta$-admissible.
\item Let $M = \mathbb{C}^3$ and $\theta$ be generated by $ \{\partial_x, \partial_y\}$. A blowing-up with center $\mathcal{C} = \{x=0, z=0\}$ is $\theta$-admissible since $\Gamma_{\theta,1}(\mathcal{I}_{\mathcal{C}}) = \mathcal{O}_M$ and $\Gamma_{\theta,2}(\mathcal{I}_{\mathcal{C}}) \subset \mathcal{I}_{\mathcal{C}}$.
\item Let $M = \mathbb{C}^3$ and $\theta$ be generated by $ \{\partial_x, \partial_y\}$. A blowing-up with center $\mathcal{C} = \{x^2-z=0,y=0\}$ is not $\theta$-admissible since $\Gamma_{\theta,2}(\mathcal{I}_{\mathcal{C}})+\mathcal{I}_{\mathcal{C}} = (x,y,z)$.
\end{enumerate}
\end{example}

\begin{remark}[Intuition of the Definition] A blowing-up $\sigma:(\widetilde{M},\widetilde{\theta},\widetilde{E}) \to (M,\theta,E)$ is $\theta$-admissible if at each point $p$ in the center $\mathcal{C}$, the singular distribution $\theta$ can be decomposed in two singular distributions germs $\theta_{inv}$ and $\theta_{tr}$ such that the center $\mathcal{C}$ is $\theta_{inv}$-invariant and $\theta_{tr}$-totally transverse (see examples 1 and 2 above).
\label{rk:IntAC}
\end{remark}
The following result enlightens the interest of $\theta$-admissible blowings-up. The rest of this section is dedicated to its proof:
\begin{proposition}
Let $(M,\theta,E)$ be a $\mathcal{K}$-monomial foliated manifold and $
\sigma: (\widetilde{M},\widetilde{\theta},\widetilde{E})\allowbreak \to (M,\theta,E)
$ be a $\theta$-admissible blowing-up. Then $\widetilde{\theta}$ is also $\mathcal{K}$-monomial.
\label{prop:AdmCenter}
\end{proposition}
\subsection{Proof of Proposition \ref{prop:AdmCenter}} The proof of the proposition relies on the following preliminary result:
\begin{lemma}
Let $\theta$ be a $\mathcal{K}$-monomial singular distribution and $\sigma$ be a $\theta$-admissible blowing-up. Then, at each point $p$ in the center of blowing-up $\mathcal{C}$, there exists a monomial coordinate system $(\pmb{u},\pmb{w})$ centered at $p$ such that $\mathcal{I}_{\mathcal{C}}$ is locally generated by $(u_1^{\epsilon_1}, \ldots, u_k^{\epsilon_k},w_{k+1}^{\epsilon_{k+1}},\ldots, w_m^{\epsilon_m})$ with $\epsilon_i \in \{0,1\}$.
\label{lem:AdmissibleCenter}
\end{lemma}
\begin{proof}
Let us fix a monomial basis $\{\partial_1,\ldots, \partial_d\}$ and a monomial coordinate system $(\pmb{u},\pmb{w})$ centered at $p$. We divide in two cases depending on the relation between $\mathcal{C}$ and $\theta$:

\medskip
\noindent
\emph{Case I: The center $\mathcal{C}$ is invariant by $\theta$:} In this case, by corollary \ref{cor:InvaBase}, there exists a system of generators $(f_1,\ldots, f_n)$ of $\mathcal{I}_{\mathcal{C}}\cdot \mathcal{O}_p$ such that $\partial_i(f_j) = \zeta_{i,j}f_j$, with $\zeta_{i,j}\in \mathcal{K}$. In particular, if any of the vector fields $\partial_i$ is regular, say $\partial_i = \partial_{w_m}$, then all functions $\zeta_{i,j}=0$ for all $j$ and $f_j$ is independent of the $w_m$ coordinate. Thus, apart from taking the quotient $\mathcal{O}_p/(\pmb{w})$, we can assume that the monomial coordinate is just $\pmb{u}$ and all vector fields $\partial_i$ are singular, i.e. they are equal to $\sum \alpha_{i,j}u_j\partial_{u_j}$.

Now, since $\mathcal{C}$ is a regular sub-variety, we can suppose that $f_1$ is regular. Thus, there exists $u_l$ such that $\partial_{u_l} f_1$ is a unit. Let $\widetilde{u}_l = f_1$ and $\widetilde{u}_i =u_i$ otherwise. After this change of coordinates we obtain:
\[
\partial_i = \sum^k_{j\neq l} \alpha_{i,j} \widetilde{u}_j \partial_{\widetilde{u}_j} + \zeta_{i,l}\widetilde{u}_l \partial_{\widetilde{u}_l} 
\]
which implies that $\widetilde{w}_l \in \mathcal{I}_{\mathcal{C}}$ and $(\pmb{\widetilde{u}},\pmb{\widetilde{w}})$ is a monomial coordinate system. Applying the above argument a finite number of times, we conclude the proof of the lemma in this case.
\medskip
\noindent
\emph{Case II: The center $\mathcal{C}$ is not invariant by $\theta$:} There exists a maximal natural number $k_0  > 0$ such that $\Gamma_{\theta,k_0}(\mathcal{I}_{\mathcal{C}})+\mathcal{I}_{\mathcal{C}} = \mathcal{O}_M$. So, without loss of generality, there exists function $(f_1,...,f_{k_0})$ in $\mathcal{I}_{\mathcal{C}}\cdot \mathcal{O}_p$ such that:
\[
 det
\left\| \begin{array}{ccc}
\partial_1(f_1) & \ldots & \partial_{1}(f_{k_0}) \\
 \vdots & \ddots & \vdots \\
 \partial_{k_0}(f_1) & \ldots & \partial_{k_0}(f_{k_0}) 
 \end{array}
\right\|  \text{ is a unit.}
\]
where $\partial_i = \partial_{w_{m+1-i}}$. Now, without loss of generality we can assume that $\partial_i(f_i)$ is a unit. We consider the change of coordinates $\widetilde{w}_{m+1-i} = f_i$ for $i\leq k_0$. After this change of coordinates we get:
\[
\begin{aligned}
\partial_i &= U_i \partial_{\widetilde{w}_{m+1-i}} + \sum^{k_0}_{j \neq i, j=1} f_{i,j} \partial_{\widetilde{w}_{m+1-j}} \text{ for } i\leq k_0
\end{aligned}
\]
for units $U_i$. So, the set of vector fields $\{\widetilde{\partial}_i\}$ given by $\widetilde{\partial}_i = \partial_{\widetilde{w}_{m+1-i}}$ for $i\leq k_0$ and $\widetilde{\partial}_i = \partial_{i} - \sum_{j=1}^{k_0} \partial_i (\widetilde{w}_{m+1-j})\partial_{\widetilde{w}_{m+1-i}}$ otherwise, is a monomial basis of $\theta$ at $p$. To finish the proof, we just need to note that, over the quotient $\mathcal{O}_p / (w_{m-k_0}, \ldots, w_m)$, we are in the same conditions of Case I.
\end{proof}
We are now ready to prove Proposition \ref{prop:AdmCenter}:
\begin{proof}[Proof of Proposition \ref{prop:AdmCenter}] Note that conditions $[i]$, $[ii]$ and $[iv]$ of the definition of monomial singular distributions are clearly true for $\widetilde{\theta}$. In order to prove that condition $[iii]$ also holds, fix a monomial basis $\{\partial_1, \ldots, \partial_d\}$. By Lemma \ref{lem:AdmissibleCenter} at each point $p$ in the center of blowing-up $\mathcal{C}$, there exists a monomial coordinate system $(\pmb{u},\pmb{w})$ centered at $p$ such that $\mathcal{I}_{\mathcal{C}}$ is locally generated by $(u_1^{\epsilon_1}, \ldots, u_k^{\epsilon_k},w_{k+1}^{\epsilon_{k+1}},\ldots, w_m^{\epsilon_m})$ with $\epsilon_i \in \{0,1\}$. We now divide the proof in two cases depending whether we consider a $u$-chart of a $w$-chart:

\medskip
\noindent
\emph{The origin of a $u$-chart:} Without loss of generality, we may suppose that $\epsilon_1=1$ and we consider the $u_1$-chart. The origin of this chart has coordinate system $(\pmb{x},\pmb{y})$ given by:
\[
u_1=x_1, \quad u_i = x_1^{\epsilon_i}x_i, \quad w_i = x_1^{\epsilon_i}y_i 
\]
In this case, the transforms of the vector fields $\partial_i$ are given by:
\[
\begin{aligned}
\partial_i &= \alpha_{i,1}x_1\partial_{x_1} + \sum_{j=2}^k (\alpha_{i,j} - \epsilon_j \alpha_{i,1})x_j \partial_{x_j}-  \sum_{j=k+1}^m \epsilon_j \alpha_{j,1} y_j \partial_{y_j}  \\
\partial_i &= \frac{1}{x_1^{\epsilon_i}}\partial_{y_i}  
\end{aligned}
\]
and it is clear that $\widetilde{\theta}$ is $\mathcal{K}$-monomial at the origin of this chart.

\medskip
\noindent
\emph{The origin of a $w$-chart:} Without loss of generality, we may suppose that $\epsilon_m=1$ and we consider the $w_m$-chart. The origin of this chart has coordinate system $(\pmb{x},\pmb{y},z)$ given by:
\[
w_m =z, \quad u_i = z^{\epsilon_i}x_i, \quad w_i = z^{\epsilon_i}y_i 
\]
In this case, the transforms of the vector fields $\partial_i$ are given by one of the following:
\[
\begin{aligned}
\partial_1 &= \frac{1}{z} \left( z\partial_z - \sum \epsilon_i x_i\partial_{x_i} + \sum \epsilon_i y_i\partial_{y_i} \right)\\
\partial_i &= \sum_{j=1}^k \alpha_{i,j}x_j \partial_{x_j}  \\
\partial_i &= \frac{1}{z^{\epsilon_i}}\partial_{y_i} 
\end{aligned}
\]
and it is clear that $\widetilde{\theta}$ is $\mathcal{K}$-monomial at the origin of this chart.

\medskip
Now, by Lemma \ref{lem:Gmonloc}, we conclude that $\widetilde{\theta}$ is monomial over all the pre-image of $p$. Since the choice of $p$ in the center $\mathcal{C}$ was arbitrary, the lemma follows.
\end{proof}
\section{Main invariants of a foliated ideal sheaf}
\label{sec:CIS}
In this section we define a pair of invariant $(\nu,type)$ just as in \cite[Section 3.1]{Bel1}. Given a foliated ideal sheaf $(M,\theta,\mathcal{I},E)$, let us consider the sequence of ideal sheaves $\left(H(\theta,\mathcal{I},n)\right)_{n \in \mathbb{N}}$ which is defined recursively as:
\[
\begin{aligned}
H(\theta,\mathcal{I},0) \phantom{+1}&:= \mathcal{I}  \\
H(\theta,\mathcal{I},n+1) &:= H(\theta, \mathcal{I},n)+ \theta[H(\theta,\mathcal{I},n)]  
\end{aligned}
\]
We now define the main pair of invariants of a foliated ideal sheaf, which was first introduced in \cite{BeloT}:
\begin{definition}
Given a foliated ideal sheaf $(M,\theta,\mathcal{I},E)$ and a point $p$ in $M$, its \textit{tg-order} $\nu =\nu(p)$ at $p$ is the smallest natural number such that $ H(\theta,\mathcal{I},\nu) = H(\theta,\mathcal{I},\nu+1)$ (which always exists since $\mathcal{O}_p$ is a noetherian ring). Moreover, the foliated ideal sheaf is of \textit{type $1$} at a point $p$ if $ H(\theta,\mathcal{I},\nu) \cdot \mathcal{O}_p = \mathcal{O}_p$ and of \textit{type $2$} otherwise. 
\label{def:inv}
\end{definition}
In what follows, we consider as main invariant the pair $(\nu,type)$ (ordered lexicographically). Note that this invariant is upper semi-continuous. If the type is one, we can work with a particular normal form: 

\begin{lemma}[Weierstrass-Thirnhaus form]
Let $(M,\theta,\mathcal{I},E)$ be a foliated ideal-sheaf and $p$ a point in $M$ where the type is $1$ and the tg-order $\nu := \nu(p)$ is positive. Then, there exists a coordinate system $(\pmb{u},v,\pmb{w})$ of $p$ and a set of generators $(g_1,\ldots,g_n)$ of $\mathcal{I}\cdot \mathcal{O}_p$ such that the vector field $\partial_v$ belongs to $\theta \cdot \mathcal{O}_p$ and:
\begin{equation}\label{eq:basicnormaform}
 \begin{aligned}
  g_1 &= v^{\nu} U + \sum^{\nu-2}_{j=0} v^j a_{1,j}(\pmb{u},\pmb{w}) \text{ where $U$ is a unit, and}\\
  g_i &= v^{\nu} \bar{g}_i + \sum^{\nu-1}_{j=0} v^j a_{i,j}(\pmb{u},\pmb{w})
 \end{aligned}
\end{equation}
\label{lem:BasicNormalForm}
\end{lemma}
\begin{proof}
By the definition of type it is clear that there exists a coordinate system $(\pmb{u},v,\pmb{w})$ of $p$ such that the vector field $\partial_v$ belongs to $\theta \cdot \mathcal{O}_p$ and, for any set of generators $(g_1,\ldots,g_n)$ of $\mathcal{I}$, without loss of generality, the function $\partial^{\nu}_{v} g_1$ is a unit. Furthermore, by the implicit function Theorem, there is a change of coordinates $(\widetilde{\pmb{u}},\widetilde{v},\widetilde{\pmb{w}}) = (\pmb{u},V(\pmb{u},v,\pmb{w}),\pmb{w})$ such that $\partial^{\nu-1}_{\widetilde{v}}g_1(\widetilde{\pmb{u}},0,\widetilde{\pmb{w}}) \equiv 0$. This clearly implies that equations \eqref{eq:basicnormaform} hold. Finally, since $ \partial_v =  U \partial_{\widetilde{v}}$ for some unit $U$, we conclude that $\partial_{\widetilde{v}}$ belongs to $\theta \cdot \mathcal{O}_p$.
\end{proof}
\section{Main Theorem \ref{thm:main2}: Overview of the proof}
\label{sec:over}
In the remainder of the article, we prove Theorem \ref{thm:main2}. We will make the ideal sheaf $\mathcal{I}$ principal by a sequence of $\theta$-admissible local blowings-up. Theorem \ref{thm:main} also follows since, by Proposition \ref{prop:AdmCenter}, $\theta$-admissible blowings-up preserve the $\mathcal{K}$-monomiality of the singular distribution $\theta$. In what follows, we always assume that the singularities of $\theta$ are log-canonical (e.g. when $\theta$ is $\mathcal{K}$-monomial). This hypothesis is only explicitly used in Lemma \ref{lem:Hir1} below and can also be removed if we consider a different kind of transform of $\theta$ instead of the strict transform (see the notion of analytic strict transform introduced in \cite{BeloT}).  Our proof of Theorem \ref{thm:main2} has three main steps:

\medskip\noindent
{\bf Step 1.} Reduction to type $1$ at every point.

\medskip
The following proposition will be proved in Section \ref{sec:type}.
\begin{proposition}
(First Step - Reduction of type, c.f \cite[Proposition 5.2]{Bel1}) Let $(M,\theta,\mathcal{I},E)$ be a $d$-foliated ideal sheaf and $p$ a point of $M$ where the type is $2$. Then, there exists a neighborhood $M_0$ of $p$ and sequence of $\theta$-admissible blowings-up of order $\geq1$:
\[
\begin{tikzpicture}
  \matrix (m) [matrix of math nodes,row sep=3em,column sep=4em,minimum width=1em]
  {(M_r,\theta_r,\mathcal{I}_r,E_r) & \cdots & (M_0,\theta_0,\mathcal{I}_0,E_0)\\};
  \path[-stealth]
    (m-1-1) edge node [above] {$\sigma_r$} (m-1-2)
    (m-1-2) edge node [above] {$\sigma_1$} (m-1-3);
\end{tikzpicture}
\]
such that the type of $(M_r,\theta_r,\mathcal{I}_r,E_r)$ is one at every point. Moreover, the tangency order $\nu$ does not increase in any point.
\label{prop:Hia}
\end{proposition}
The above result is first proved in \cite{BeloT}. For completeness, we provide a proof in here but we follow the more direct approach of \cite[Proposition 5.2]{Bel1}, where the theorem is proved under the additional hypothesis that the singular distributions $\theta$ has leaf dimension one (the proof in here is, mutatis mutandis, the same of \cite[Proposition 5.2]{Bel1}). The result is a consequence of the functoriality of resolution of singularities (see, e.g \cite{BM2}, for a statement of the functorial property).

Note that a point where the type is one satisfies the conclusion of Lemma \ref{lem:BasicNormalForm}. In this case, we will say that $(M,\theta,\mathcal{I},E)$ is in \textit{Weierstrass-Thirnhaus form} at $p$, i.e. there exists a coordinate system $(\pmb{u},v,\pmb{w})$ of $p$ and a set of generators $(g_1,\ldots,g_n)$ of $\mathcal{I}\cdot \mathcal{O}_p$ such that the vector field $\partial_v$ belongs to $\theta \cdot \mathcal{O}_{p}$ and the functions $g_i$ respect equation \eqref{eq:basicnormaform}.

\medskip\noindent
{\bf Step 2.} Reduction to prepared normal form.

\medskip
The following proposition will be proved in Section \ref{sec:prep}. 
 
\begin{proposition}[Second Step - Prepared normal form]\label{prop:prepnormal}
Suppose that the type of $(M,\theta,\mathcal{I},E)$ is $1$ at every point and that $\nu$ takes a maximal value $\nu_{max}>0$ at $M$. Let $\Sigma$ be the maximal locus set of $\nu$. Furthermore, suppose that Theorem $\ref{thm:main2}$ holds for any foliated ideal-sheaf $(N,\omega,\mathcal{J},F)$ with $dim N < dim M$. Then, by a finite collection of local $\theta$-admissible blowings-up, we can reduce to the case that, for every point $q\in \Sigma$, the foliated ideal sheaf $(M,\theta,\mathcal{I},E)$ has Weierstrass-Thirnhaus form of Lemma \ref{lem:BasicNormalForm}, where the coefficients $a_{i,j}$ satisfies the following additional condition:
\[
a_{i,j}(\pmb{u},\pmb{w}) = \pmb{u}^{\pmb{r}_{i,j}} b_{i,j}(\pmb{u},\pmb{w})
\]
where $b_{i,j}(\pmb{u},\pmb{w})$ is either a unit (and $\pmb{r}_{i,j} \neq 0$) or zero. Finally, the blowings-up involved do not increase the value of $\nu$ over any point.
\end{proposition}
\begin{remark}
Note that, if $dim M =1$, the inductive hypothesis \textit{``Theorem $\ref{thm:main2}$ is true any foliated ideal-sheaf $(N,\omega,\mathcal{J},F)$ with $dim N < dim M$''} is trivially true.
\end{remark}
In case the foliated ideal sheaf $(M,\theta,\mathcal{I},E)$ satisfies the thesis of Proposition \ref{prop:prepnormal} at a point $p$, we will say that $p$ is a \textit{prepared} point and that the Weierstrass-Thirnhaus form in Lemma \ref{lem:BasicNormalForm} is prepared at $p$.

\medskip\noindent
{\bf Step 3.} Further admissible blowings-up to decrease the maximal value of the invariant $\nu$.

\medskip
The following Proposition \ref{prop:Dropping} will be proved in Section \ref{sec:Drop}.
\begin{proposition}[Third Step - Dropping the tg-order]
\label{prop:Dropping}
Suppose that $(M,\theta,\mathcal{I},E)$ satisfies the prepared normal form (see Proposition \ref{prop:prepnormal}) at a point $p$. Then, for a small enough neighborhood $M_0$ of $p$, there exists a sequence of $\theta$-admissible blowings-up $\pmb{\tau}: (M_r,\theta_r,\mathcal{I}_r,E_r) \rightarrow (M_0,\theta_0,\mathcal{I}_0,E_0)$ such that, for all point $q$ in $M_r$, the invariant $\nu(q)$ is strictly smaller than $\nu(p)$.
\end{proposition}

We can now prove the main result of this work:

\begin{proof}[Proof of Theorem \ref{thm:main2}]
The proof of the theorem follows from Propositions \ref{prop:Hia}, \ref{prop:prepnormal} and \ref{prop:Dropping} by induction on the dimension of $M$ and the maximal value of the invariant $\nu$. In particular, note that when $(\nu,type) = (0,0)$, we conclude that the pull-back of $\mathcal{I}$ is principal.
\end{proof}

\section{Main Theorem \ref{thm:main2} Dropping the type}
\label{sec:type}

We follow the notation of section \ref{sec:over} and we prove Proposition \ref{prop:Hia} in this section. The result was first proved in \cite{BeloT} and a more direct proof, under the additional hypothesis that the leaf dimension of $\theta$ is one, is given in \cite[Proposition 5.2]{Bel1}. For completeness, we will follow \cite{Bel1} and we will make the necessary adaptations. 

The main preliminary result which is necessary to prove the proposition is the following:

\begin{lemma}(c.f \cite[Proposition 4.1]{Bel1}) 
\label{lem:HironakaS}
Let $(M,\theta, \mathcal{I},E)$ be a foliated ideal sheaf and $M_0$ a relatively compact open set of $M$. Suppose that $\mathcal{I}_0:=\mathcal{I} \cdot \mathcal{O}_{M_0}$ is invariant by $\theta_0:=\theta \cdot \mathcal{O}_{M_0}$, i.e., $\theta_0 [ \mathcal{I}_0]\allowbreak \subset \mathcal{I}_0$. Then, there exists a sequence of $\theta$-admissible blowings-up of order $\geq1$:
\[
\begin{tikzpicture}
  \matrix (m) [matrix of math nodes,row sep=3em,column sep=3em,minimum width=1em]
  {(\widetilde{M},\widetilde{\theta},\widetilde{\mathcal{I}},\widetilde{E}) = (M_r,\theta_r,\mathcal{I}_r,E_r) & \cdots & (M_0,\theta_0,\mathcal{I}_0,E_0)\\};
  \path[-stealth]
    (m-1-1) edge node [above] {$\sigma_r$} (m-1-2)
    (m-1-2) edge node [above] {$\sigma_1$} (m-1-3);
\end{tikzpicture}
\]
such that $\widetilde{\mathcal{I}} =\mathcal{O}_{M}$.
\end{lemma}

This result is a consequence of the functoriality of resolution of singularities. To finish the proof of Proposition \ref{prop:Hia}, we will just need to use the above lemma over $H(\theta,\mathcal{I},\nu)$ and make the necessary computations in order to control the pair of invariant $(\nu,type)$ (mutatis mutandis, the same argument of \cite[section 5.3]{Bel1}).

\subsection{Proof of Lemma \ref{lem:HironakaS}} 
In order to prove the lemma, we introduce the notion of geometrical invariance: an ideal sheaf $\mathcal{I}$ is \textit{geometrically invariant} by $\theta$ if every leaf of the foliation generated by $\theta$ that intersects $V(\mathcal{I})$ is totally contained in $V(\mathcal{I})$.
\begin{lemma}(c.f \cite[Lemma 4.2]{Bel1})  
Let $\theta$ be an involutive singular distribution and $\mathcal{I}$ a reduced ideal sheaf. Then $\mathcal{I}$ is $\theta$-invariant if, and only if, $\mathcal{I}$ is geometrically invariant by $\theta$.
\label{lem:Fol}
\end{lemma}
\begin{proof}
To prove the if implication, let $p$ be a point in $V(\mathcal{I})$, $\mathcal{L}$ be the leaf of $\theta$ passing through $p$ and $\{\partial_1, \ldots, \partial_d\}$ be vector fields that generate $\theta\cdot \mathcal{O}_p$. Without loss of generality, we can assume that there exists a number $0 \leq t \leq d$ and a coordinate system $(\pmb{x},\pmb{y}) = (x_1, \ldots, x_t, y_{t+1}, \ldots, y_m)$ such that $\partial_i = \partial_{x_i}$ for $i\leq t$ and $\partial_i$ is singular for $i> t$. In this case, it is clear that there exists a set of generators $(f_1(\pmb{y}), \ldots, f_r(\pmb{y}))$ which are independent of the $\pmb{x}$ coordinates and the result is now obvious.

To prove the only if implication, let us assume that $\mathcal{I}$ is a reduced ideal sheaf which is geometrically invariant by $\theta$. We claim that $V(\mathcal{I}) \subset V(\theta(\mathcal{I}))$, which is enough to prove the lemma since it would imply that:
\[
\theta(\mathcal{I}) \subset \sqrt{\theta(\mathcal{I})} \subset \sqrt{\mathcal{I}} = \mathcal{I}
\]
In order to prove the claim, let $p$ be a point of $ V(\mathcal{I})$, $\mathcal{L}$ be the leaf of $\theta$ passing through $p$, $f$ be an arbitrary function in $\mathcal{I}$ and $\{\partial_1, \ldots, \partial_d \}$ be vector fields that generates $\theta\cdot \mathcal{O}_p$. Since, by hypothesis, $\mathcal{L} \subset V(\mathcal{I})$, we conclude that $f|_{\mathcal{L}} \equiv 0$. Moreover, since $\theta$ is tangent to $\mathcal{L}$, we conclude that $\partial_i(f)|_{\mathcal{L}} = \partial_i|_{\mathcal{L}}(f|_{\mathcal{L}}) =0$ for any $i$ and, in particular, $p \in V(\theta(f))$. Since the choice of $f \in \mathcal{I}$ was arbitrary, we conclude that $p$ belongs to $V(\theta(\mathcal{I}))$ as we wanted to prove.
\end{proof}
We now state a preliminary lemma:
\begin{lemma}(see \cite[Lemma 4.3]{Bel1}) 
Let $(M,\theta,\mathcal{I},E)$ be a foliated ideal sheaf and let us consider a $\theta$-invariant blowing-up of order $\geq1$ $\sigma: (\widetilde{M},\widetilde{\theta},\widetilde{\mathcal{I}},\widetilde{E}) \longrightarrow (M,\theta,\mathcal{I},E)$. Then $\widetilde{\mathcal{I}}$ is invariant by $\widetilde{\theta}$.
\label{lem:Hir1}
\end{lemma}
\begin{proof}
In order to be self-contianed, we reproduce the proof of \cite[Lemma 4.3]{Bel1}. Since the blowing-up is $\theta$-invariant, the ideal sheaf $\mathcal{I}_{\mathcal{C}}$ is $\theta$-invariant. Thus:
\[
\begin{array}{ccc}
\theta [ \mathcal{I}_{\mathcal{C}}] \subset \mathcal{I}_{\mathcal{C}} & \Rightarrow & \theta^{\ast}[\mathcal{I}_{F}] \subset \mathcal{I}_{F}
\end{array}
\]
Now, since $\theta$ has only log-canonical singularities, $\widetilde{\theta}=\sigma^{\ast}\theta$ and we conclude that:
\[
\begin{aligned}
\widetilde{\theta}[\widetilde{\mathcal{I}}] + \widetilde{\mathcal{I}} &= \theta^{\ast}[\mathcal{I}^{\ast} \cdot \mathcal{I}_F^{-1}] + \widetilde{\mathcal{I}} \\
&\subset \theta^{\ast}[ \mathcal{I}^{\ast}] \mathcal{I}_F^{-1} + \mathcal{I}^{\ast} \theta^{\ast}[\mathcal{I}_F^{-1}] + \widetilde{\mathcal{I}} = \widetilde{\mathcal{I}}
\end{aligned}
\]
which concludes the lemma.
\end{proof}
\noindent
Now, we are ready to start the proof of Lemma \ref{lem:HironakaS} (c.f proof of \cite[Lemma 4.1]{Bel1}).
\begin{proof}[Proof of Lemma \ref{lem:HironakaS}]
By the usual Hironaka's resolution of singularities Theorem, there exists a resolution of singularities of $(M_0,\theta_0,\mathcal{I}_0,E_0)$:
\[
\begin{tikzpicture}
  \matrix (m) [matrix of math nodes,row sep=3em,column sep=3em,minimum width=1em]
  { (\widetilde{M},\widetilde{\theta},\widetilde{\mathcal{I}},\widetilde{E})=(M_r,\theta_r,\mathcal{I}_r,E_r) & \cdots & (M_0,\theta_0,\mathcal{I}_0,E_0)\\};
  \path[-stealth]
    (m-1-1) edge node [above] {$\sigma_r$} (m-1-2)
    (m-1-2) edge node [above] {$\sigma_1$} (m-1-3);
\end{tikzpicture}
\]
where $\sigma_i : (M_i,\theta_i,\mathcal{I}_i,E_i) \longrightarrow (M_{i-1},\theta_{i-1},\mathcal{I}_{i-1},E_{i-1})$ has center $\mathcal{C}_i$. We claim that every center $\mathcal{C}_i$ is $\theta_{i-1}$ invariant, which finishes the proof. We prove the claim by induction: Suppose that the centers $\mathcal{C}_i$ are $\theta_{i-1}$-invariant for $i<k$. We need to verify that $\mathcal{C}_k$ is also $\theta_{k-1}$-invariant (including for $k=1$). Since $\mathcal{C}_k$ is regular, by Lemma $\ref{lem:Fol}$, we only need to verify that $\mathcal{C}_k$ is geometrically invariant by $\theta_{k-1}$.

To this end, let $\mathcal{L}$ be a connected leaf of $\theta_{k-1}$ with non-empty intersection with $\mathcal{C}_{k}$. Fix $p \in \mathcal{L}\cap \mathcal{C}_k$ and let $\{\partial_1. \ldots, \partial_d\}$ be vector fields that generate $\theta\cdot \mathcal{O}_p$. Without loss of generality, we can assume that there exists a number $0 \leq t \leq d$ and a coordinate system $(\pmb{x},\pmb{y}) = (x_1, \ldots, x_t, y_{t+1}, \ldots, y_m)$ centered at $p$ such that $\partial_i = \partial_{x_i}$ for $i\leq t$ and $\partial_i$ is singular for $i> t$. In particular, $\{\partial_1, \ldots, \partial_t\}$ locally generates $\mathcal{L}$. Now, note that $\mathcal{I}_{k-1}$ is $\theta_{k-1}$-invariant because of the induction hypotheses and a recursive use Lemma \ref{lem:Hir1}. So, it is easy to see that there exists a set $\{f_1(\pmb{y}),...,f_n(\pmb{y})\}$ of local generators of $\mathcal{I}_{k-1}. \mathcal{O}_p$ which is independent of the $\pmb{x}$ coordinates. By the functoriality of Hironaka`s resolution of singularity Theorem, the center $\mathcal{C}_k$ is locally geometrically invariant by $\theta$. Since the choice of $p$ in the intersection $\mathcal{C}_k \cap \mathcal{L}$ was arbitrarily, we conclude that $\mathcal{L} \subset \mathcal{C}_k$, which ends the proof.
\end{proof}
\subsection{Proof of Proposition \ref{prop:Hia}}
Let $M_0$ be a sufficiently small relatively compact neighborhood of $p$ where the maximal value of the tg-order is equal to $\nu:=\nu(p)$. By hypothesis, $H(\theta_0,\mathcal{I}_0,\nu)$ is a $\theta$-invariant ideal-sheaf. So, set 
$\mathcal{C}l:=H(\theta,\mathcal{I},\nu)$ and, by Lemma $\ref{lem:HironakaS}$, there exists a $\theta$-invariant sequence of blowings-up of order $\geq1$:
\[
\begin{tikzpicture}
  \matrix (m) [matrix of math nodes,row sep=3em,column sep=3em,minimum width=2em]
  {(\widetilde{M},\widetilde{\theta},\widetilde{\mathcal{C}l},\widetilde{E}) =(M_r,\theta_r,\mathcal{C}l_r,E_r) & \cdots  & (M_0,\theta_0,\mathcal{C}l_0,E_0)\\};
  \path[-stealth]
    (m-1-1) edge node [above] {$\sigma_r$} (m-1-2)
    (m-1-2) edge node [above] {$\sigma_1$} (m-1-3);
\end{tikzpicture}
\]
such that $\widetilde{\mathcal{C}l} = \mathcal{O}_{\widetilde{M}}$. Furthermore:
\begin{claim}
 The sequence of blowings-up $\vec{\sigma}$ is of order $\geq1$ in respect to $\mathcal{I}_0$ and $\mathcal{C}l_{j} = H(\theta_j,\mathcal{I}_j,\nu)$ for all $j  \leq r$.
 \label{cl:1a}
\end{claim}
\noindent
If we assume the claim, the result is now obvious. Indeed, if we take the same sequence of blowing-up in respect to $(M_0,\theta_0, \allowbreak \mathcal{I}_0,E_0)$, we obtain a $\theta$-invariant sequence of blowings-up of order $\geq1$:
\[
\begin{tikzpicture}
  \matrix (m) [matrix of math nodes,row sep=3em,column sep=3em,minimum width=2em]
   {(\widetilde{M},\widetilde{\theta},\widetilde{\mathcal{I}},\widetilde{E}) =(M_r,\theta_r,\mathcal{I}_r,E_r) & \cdots  & (M_0,\theta_0,\mathcal{I}_0,E_0)\\};
  \path[-stealth]
    (m-1-1) edge node [above] {$\sigma_r$} (m-1-2)
    (m-1-2) edge node [above] {$\sigma_1$} (m-1-3);
\end{tikzpicture}
\]
such that $ H(\widetilde{\theta},\widetilde{\mathcal{I}},\nu) = \widetilde{\mathcal{C}l} = \mathcal{O}_{\widetilde{M}}$, which is the desired result. Now, in order to prove the claim, the main step is the following lemma:
\begin{lemma}
Let $\sigma: (\widetilde{M},\widetilde{\theta},\widetilde{\mathcal{I}},\widetilde{E}) \longrightarrow (M,\theta,\mathcal{I},E)$ be a $\theta$-invariant blowing-up with center contained in $V(\mathcal{C}l)$, where $\mathcal{C}l:=H(\theta,\mathcal{I},\nu)$. Then
\[
H(\widetilde{\theta},\widetilde{\mathcal{I}}, i) = H(\theta,\mathcal{I}, i)^{\ast}  \mathcal{I}_F^{-1}
\]
for every $i  \leq \nu$, where $F$ is the exceptional divisor created by $\sigma$. In particular $\widetilde{\mathcal{C}l} = H(\widetilde{\theta},\widetilde{\mathcal{I}},\nu)$.
\label{lem:25}
\end{lemma}
\begin{proof}
First, note that, since $H(\theta,\mathcal{I}, i) \subset \mathcal{C}l$ for $i \leq \nu$, the center of blowing-up is also contained in $V(H(\theta,\mathcal{I},i))$ for all $i\leq \nu$. Now, if $\mathcal{J}$ is a coherent ideal sheaf, then:
\[
\widetilde{\theta} [\mathcal{I}_{F}] \subset \mathcal{I}_{F} \Rightarrow \mathcal{J} \widetilde{\theta}[\mathcal{I}_F^{-1}] \subset  \mathcal{J} \mathcal{I}_F^{-1}
\]
In particular, this implies that:
\[
\widetilde{\theta}[\mathcal{J}\mathcal{I}_F^{-1}] + \mathcal{J}\mathcal{I}_F^{-1} = \mathcal{I}_F^{-1}(\widetilde{\theta}[\mathcal{J}] +\mathcal{J})
\]
Now, it rests to prove that the following equality $H(\widetilde{\theta},\widetilde{\mathcal{I}}, i) =  H(\theta,\mathcal{I}, i)^{\ast} \mathcal{I}_F^{-1}
$ is valid for all $i \leq \nu$. Indeed, suppose by induction that the equality is valid for $i < k$ (note that for $k=0$, the equality is trivial). Since the center of blowing-up is contained in $V(H(\theta,\mathcal{I}))$ we have that:
\[
\begin{aligned}
H(\widetilde{\theta},\widetilde{\mathcal{I}}, k) &= H(\widetilde{\theta},\widetilde{\mathcal{I}}, k-1) + \widetilde{\theta}[H(\widetilde{\theta}, \widetilde{\mathcal{I}},k-1)]\\
&=H(\widetilde{\theta},\widetilde{\mathcal{I}}, k-1) + \theta^{\ast}[ H(\theta,\mathcal{I}, k-1)^{\ast} \mathcal{O}( F)] \\
&=\mathcal{I}_F^{-1} \{ H(\theta,\mathcal{I}, k-1)^{\ast} +\theta^{\ast}[H(\theta,\mathcal{I}, k-1)^{\ast}] \} \\
&=\mathcal{I}_F^{-1} \{ H(\theta,\mathcal{I}, k-1) +\theta[H(\theta,\mathcal{I}, k-1)] \}^{\ast} \\
&=\mathcal{I}_F^{-1} H(\theta,\mathcal{I}, k)^{\ast}
\end{aligned}
\]
which proves the equality and the lemma.
\end{proof}
\noindent
We now turn to the proof of the Claim \ref{cl:1a}:
\begin{proof}[Proof of Claim \ref{cl:1a}]
Suppose by induction that for $i<k$, the sequence of blowings-up $(\sigma_i , \ldots , \sigma_1)$ is of order $\geq1$ in respect to $(M_0,\theta_0,\mathcal{I}_0,E_0)$ and that $\mathcal{C}l_{i} = H(\theta_i,\mathcal{I}_i,\nu)$. Let us prove the result $i=k$ (including $k=1$). Since the center of $\sigma_k$ is contained in $V(\mathcal{C}l_{k-1})$, by the induction hypotheses it is also contained in $V(\mathcal{I}_{k-1})$, which implies that the  sequence of blowings-up $(\sigma_k , \ldots , \sigma_1)$ is of order $\geq1$ in respect to $(M_0,\theta_0,\mathcal{I}_0,E_0)$. Furthermore, by Lemma $\ref{lem:25}$ and the induction hypotheses:
\[
H(\theta_k,\mathcal{I}_k,\nu) = H(\theta_{k-1},\mathcal{I}_{k-1},\nu)^{\ast} \, \mathcal{I}_F^{-1} = [\mathcal{C}l_{k-1}]^{\ast} \, \mathcal{I}_F^{-1} = \mathcal{C}l_k
\]
which finishes the proof.
\end{proof}

\section{Main Theorem \ref{thm:main2} Preparation}
\label{sec:prep}
We follow the notation of section \ref{sec:over} and we prove Proposition \ref{prop:prepnormal} in this section. By Lemma \ref{lem:BasicNormalForm}, the analytic $d$-foliated ideal-sheaf $(M,\theta,\mathcal{I},E)$ satisfies the Weierstrass-Thirnhaus form at $p$, i.e. there exists a coordinate system $(\pmb{u},v,\pmb{w})$ of $p$ and a set of generators $(g_1,\ldots,g_n)$ of $\mathcal{I}\cdot \mathcal{O}_p$ such that the vector field $\partial_v$ belongs to $\theta \cdot \mathcal{O}_{p}$ and the functions $g_i$ are given by equation \eqref{eq:basicnormaform}.

The main idea is to make blowings-up in order to monomialize the coefficients $a_{i,j}$ of equation \eqref{eq:basicnormaform} without changing the value of the main invariant $\nu$. In order to accomplish this (and to avoid the problem presented in example \ref{ex:main}), all the centers of blowing-up will be independent of the $v$-coordinate. Note that the axis $\partial_v$ is a priori only locally defined and it does not need to have a global extension.

So, consider the locally defined principal ideal $\widetilde{\mathcal{J}}$ generated by the product of all non-zero $a_{i,j}$ and a projection map $\pi: M_0 \rightarrow N$ given by $\pi(\pmb{u},v,\pmb{w}) = (\pmb{u},\pmb{w})$, where $M_0$ is a neighborhood of $p$ where $\widetilde{\mathcal{J}}$ is well-defined. There exist a $d-1$ foliated ideal sheaf $(N,\omega,\mathcal{J},F)$ such that:
\begin{itemize}
 \item The singular distribution $\theta$ is generated by $\pi^{\ast}\omega$;
\item The ideal sheaf $\widetilde{\mathcal{J}}$ is equal to $\pi^{\ast}\mathcal{J}$;
 \item The inverse image of $F$ is equal to $E \cap M_0$.
\end{itemize}
Since $dim N < dim M$, by the induction hypothesis there exists a finite collection of $\omega$-admissible local blowings-up $\sigma_i : (N_i,\omega_i,\mathcal{J}_i,F_i) \rightarrow (N,\omega,\mathcal{J},F)$ such that:
\begin{itemize}
 \item The morphism $\sigma_i$ is a finite composition of $\omega$-admissible local blowing-ups;
 \item In each variety $N_i$, there exists a compact set $V_i \subset N_i$ such that the union of their images $\bigcup \sigma_i(V_i)$ is a compact neighborhood of $\pi(p)$;
 \item The ideal sheaf $\mathcal{J}_i$ is the structural ring. In particular, this means that the total transform $\sigma_i^{\ast} \mathcal{J}$ is a principal ideal sheaf with support contained in $F_i$.
\end{itemize}
Furthermore, it is clear that we can extend $\sigma_i$ to blowings-up at $M_0$ by taking the product of the centers of $\tau_i$ by the $v$-axis:
\[
 \tau_i (M_i,\theta_i,\mathcal{I}_i,E_i) \to (M_0,\theta_0,\mathcal{I}_0,E_0)
\]
where all centers have SNC with the exceptional divisor and are invariant by the $v$-coordinate i.e. all centers are $\partial_v$-invariant. Since all centers of $\sigma_i$ are $\omega$-admissible, we conclude that all centers of $\tau_i$ are $\theta$-admissible.

Note that no center is contained in the support of the ideal sheaf $\mathcal{I}$, which implies that $\mathcal{I}_i$ is the total transform of $\mathcal{I}$. Consider $q_i$ a point in the pre-image of $p$ by $\tau_i$ and let $(\pmb{x},y,\pmb{z})$ be a coordinate system at $q_i$ such that $\tau_i^{\ast}v=y$. Since the pull-back $(\tau_i \circ \pi)^{\ast} \mathcal{J}$ is a principal ideal sheaf, we conclude that equation $\ref{eq:basicnormaform}$ transforms to:

\[
 \begin{aligned}
  g_1 &= y^{\nu} U + \sum^{\nu-2}_{j=0} y^j \pmb{x}^{r_{1,j}} \bar{a}_{1,j}(\pmb{x},\pmb{z}) \text{ where $U$ is a unity, and}\\
  g_i &= y^{\nu} \bar{g}_i + \sum^{\nu-1}_{j=0} y^j \pmb{x}^{r_{i,j}} \bar{a}_{i,j}(\pmb{x},\pmb{z})
 \end{aligned}
\]
where the functions $\bar{a}_{i,j}$ are either zero or units. Finally, since $\partial_{y} = \tau^{\ast}_i \partial_v$, we conclude that $\partial_{y} \in \theta_i \cdot \mathcal{O}_{q_i}$. In particular, this implies that $\nu(q_i) \leq \nu(p)$, which concludes the proof of the proposition.
\section{Main Theorem \ref{thm:main2}: Dropping the invariant}
\label{sec:Drop}
We follow the notation of section \ref{sec:over} and we prove Proposition \ref{prop:Dropping} in this section. We start by a couple of preliminary results about the combinatorial blowings-up.
\subsection{Combinatorial blowings-up}
\begin{definition}[Sequence of combinatorial blowings-up]\label{def:combblowingsup}
Given a divisor $E$ in $M$, we say that $\tau:\widetilde{M}\to M$ is a sequence of combinatorial blowings-up (with respect to $E$) if $\tau$ is a composition of blowings-up with centers that are strata of the divisor $E$ and its total transforms.
\end{definition}

Given a globally defined system of coordinates $(\pmb{u},v,\pmb{w})$, let us consider a sequence of combinatorial blowings-up, $\pmb{\tau}:(\widetilde{M},\widetilde{E})\to (M,E)$, with respect to the declared exceptional divisor $F=E \cup \{v=0\} = \{u_1\dotsm u_l \cdot v =0\}$.

\begin{remark}
If $(M,\theta,E)$ is a $d$-foliated manifold such that $\partial_v$ is a vector field in $\theta$, then a combinatorial sequence of blowings-up (in respect to $F=E \cup \{v=0\}$) is $\theta$-admissible. Indeed, the singular distribution can be decomposed in $\{\omega,\partial_v\}$ where $F$ is invariant by $\omega$. This clearly implies that all blowings-up are $\omega$-invariant. To conclude, note that all blowings-up are either $\partial_v$ invariant or $\partial_v$ transverse since all centers of blowing-up have simple normal crossings with $(v=0)$ and its total transform.
\label{rk:combinatorial}
\end{remark}

In the rest of this subsection we find special coordinate systems at all points $q$ in the exceptional divisor $\widetilde{E}$. First of all, we can cover $\widetilde{M}$ by affine charts with a coordinate system $(\bar{\pmb{x}},\pmb{w})$ satisfying:
\begin{equation}
\label{eq:blowingsup}
\left\{
\begin{aligned}
u_i &= \bar{x}_1^{a_{i,1}} \cdots  \bar{x}_{l+2}^{a_{i,l+1}}, \\
v\phantom{_j} &= \bar{x}_1^{a_{l+1,1}} \cdots  \bar{x}_{l+1}^{a_{l+1,l+1}},\\\notag
w_j &= w_j,
\end{aligned}
\begin{aligned}
i &= 1, \ldots, l\\
\\
j &= l+2, \ldots m
\end{aligned}
\right.
\end{equation}
that we denote, to simplify notation, by:
\[
(\pmb{u},v,\pmb{w}) = (\bar{\pmb{x}}^{\boldsymbol{A}},\pmb{w})
\]
where $\boldsymbol{A}$ is a $(l+1)$-square matrix such that  $det(\boldsymbol{A}) = \pm 1$, given by:
\[
\boldsymbol{A} =  
\begin{bmatrix}
 a_{1,1} & \ldots  & a_{1,l+1} \\
 \vdots &  \ddots & \vdots \\
a_{l+1,1} & \ldots  & a_{l+1,l+1} 
\end{bmatrix}
\]
Given a point $q$ contained in $\widetilde{E}$ and in this affine chart, we consider the following coordinate system centered at $q$:
\[
(\widetilde{\pmb{x}},\widetilde{\pmb{y}},\pmb{w}) = (\widetilde{x}_1,\ldots ,\widetilde{x}_t,\widetilde{y}_{t+1} ,\ldots , \widetilde{y}_{l+1} ,\pmb{w}),
\]
where, apart from re-ordering $\bar{\pmb{x}}$, we can assume that $\widetilde{x}_i = \bar{x}_i$ for $i=1, \ldots t$ and $\widetilde{y}_i = \bar{x}_i - \widetilde{\gamma}_i$ for $i=t+1, \ldots, l+1$ and $\widetilde{\gamma}_i \neq 0$. Note that $t \neq 0$ since $q$ is in the exceptional divisor $\widetilde{E}$. Now, we consider a decomposition of the matrix $\boldsymbol{A}$:
\[
\boldsymbol{A} = 
\begin{bmatrix}
 \boldsymbol{A}_1&\boldsymbol{A}_2 \\
 \boldsymbol{\alpha}_1&\boldsymbol{\alpha}_2
\end{bmatrix}
\]
\noindent
where $\boldsymbol{A}_1$ is a $l \times t$ matrix, $\boldsymbol{\alpha}_1$ is a $1\times t$ matrix, $\boldsymbol{A}_2$ is a $l\times (l+1-t)$ matrix and $\boldsymbol{\alpha}_2$ is a $1\times (l+1-t)$ matrix. In this notation, we can write
\begin{equation}\label{eq:baseeq}
\begin{aligned}
\pmb{u}&=\widetilde{\pmb{x}}^{\pmb{A}_1}(\widetilde{\pmb{y}}-\widetilde{\pmb{\gamma}})^{\pmb{A}_2}\\
v&=\widetilde{\pmb{x}}^{\pmb{\alpha}_1}(\widetilde{\pmb{y}}-\widetilde{\pmb{\gamma}})^{\pmb{\alpha}_2}\\
\pmb{w}&=\pmb{w}
\end{aligned}
\end{equation}
where $\widetilde{\pmb{\gamma}} = (\widetilde{\gamma}_{t+1}, \ldots, \widetilde{\gamma}_{l+1})$. We divide our analysis in two cases depending on the rank of $\boldsymbol{A}_1$:

\begin{lemma}[Case 1]\label{lem:claim1}
Assume $\pmb{A}_1$ has maximal rank (equals to $t$). There exists a vector $\pmb{\gamma} = (\gamma_{t+1},\ldots, \gamma_l)$ with non-zero entries, a non-zero constant $\gamma_{l+1}$, and a coordinate system $(\pmb{x},\pmb{y},z,\pmb{w}) = (x_1,\ldots, x_t,y_{t+1},\ldots ,\allowbreak y_{l},z,\pmb{w})$ centered at $q$ whose relation with $(\pmb{u},v,\pmb{w})$ is given by:
\begin{equation}\label{eq:claim1}
\begin{aligned}
\pmb{u}&=\pmb{x}^{\pmb{A}_1}(\pmb{y}-\pmb{\gamma})^{\boldsymbol{\Lambda}}\\
v&=\pmb{x}^{\boldsymbol{\alpha}_1}(z-\gamma_{l+1})\\
\pmb{w}&=\pmb{w}
\end{aligned}
\end{equation}
where $\pmb{\Lambda}$ is a $t\times (l-t)$ matrix equal to $\left[\begin{matrix} \pmb{0} \\ \pmb{Id} \end{matrix}\right]$ and $\pmb{Id}$ is the $(l-t)$ identity matrix. In particular, if $\mathcal{X}$ is the distribution generated by $\partial_v$, then the strict transform of $\mathcal{X}$ is generated by $\partial_z $. 

\end{lemma}

\begin{lemma}[Case 2]\label{lem:claim2}
Assume that $\pmb{A}_1$ does not have maximal rank. There exists a vector $\pmb{\gamma} = (\gamma_{t+1},\ldots, \gamma_l,\gamma_{l+1})$ with non-zero entries, a maximal rank $(l \times l+1-t)$ matrix $\pmb{\Lambda}$ (with coefficients in $\mathbb{Q}$) and a coordinate system $(\pmb{x},\pmb{y},\pmb{w}) = (x_1,\ldots, x_t,y_{t+1}, \allowbreak \ldots  ,  y_{l+1},\pmb{w})$ centered at $q$ whose relation with $(\pmb{u},v,\pmb{w})$ is given by:
\begin{equation}\label{eq:claim2}
\begin{aligned}
\pmb{u}&=\pmb{x}^{\boldsymbol{A}_1}(\pmb{y}-\boldsymbol{\gamma})^{\boldsymbol{\Lambda}}\\
v&=\pmb{x}^{\boldsymbol{\alpha}_1}\\
\pmb{w}&=\pmb{w}
\end{aligned}
\end{equation}
In particular, the vector $\pmb{\alpha}_1$ doesn't belong to the span of the rows of $\pmb{A}_1$.
\end{lemma}
\begin{proof}[Proof of Lemma \ref{lem:claim1}]
Apart from re-indexing of the $u_j$'s we can assume that 
\[
\boldsymbol{A}_1=\begin{bmatrix}\boldsymbol{A}_1'\\\boldsymbol{A}_1''\end{bmatrix},\ \boldsymbol{A}_2=\begin{bmatrix}\boldsymbol{A}_2'\\\boldsymbol{A}_2''\end{bmatrix},
\]
where $\pmb{A}_1'$ is a $t\times t$ matrix such that $\text{det}(\boldsymbol{A}_1')\neq0$, and $\boldsymbol{A}_2'$ is a $t\times (l+1-t)$ matrix. Denoting $\pmb{u}$ by $(\pmb{u}',\pmb{u}'')$ accordingly, we can write equations \eqref{eq:baseeq} in the compact form
\[
\begin{aligned}
\pmb{u}'\phantom{'}&=\widetilde{\pmb{x}}^{\boldsymbol{A}_1'}(\widetilde{\pmb{y}}-\widetilde{\boldsymbol{\gamma}})^{\boldsymbol{A}_2'}\\
\pmb{u}''&=\widetilde{\pmb{x}}^{\boldsymbol{A}_1''}(\widetilde{\pmb{y}}-\widetilde{\boldsymbol{\gamma}})^{\boldsymbol{A}_2''}\\
v\phantom{''}&=\widetilde{\pmb{x}}^{\boldsymbol{\alpha}_1}(\widetilde{\pmb{y}}-\widetilde{\boldsymbol{\gamma}})^{\boldsymbol{\alpha}_2}
\end{aligned}
\]
Now, since $det(\pmb{A}_1') \neq0$ and $det(\pmb{A}) \neq0$, there exists a coordinate system $(\pmb{x},\pmb{y},z)$, a vector $\boldsymbol{\gamma} = (\gamma_{t+1}, \ldots, \gamma_l)$ with non-zero entries and a non-zero constant $\gamma_{l+1}$ such that:
\begin{equation}\label{eq:afterchangecoord2}
\begin{aligned}
\pmb{u}'\phantom{'}&=\pmb{x}^{\boldsymbol{A}_1'}\\
\pmb{u}''&=\pmb{x}^{\boldsymbol{A}_1''}(\pmb{y}-\boldsymbol{\gamma})^{\boldsymbol{Id}}\\\notag
v\phantom{''}&=\pmb{x}^{\boldsymbol{\alpha}_1}(z-\gamma_{l+1})
\end{aligned}
\end{equation}
where $\boldsymbol{Id}$ is the $l-t$ identity matrix. To finish, note that the pull-back of $v \partial_v$ is given by:
\[
 v \partial_v = (\widetilde{z}-\widetilde{\gamma}_{l+1})\partial_{\widetilde{z}}
\]
which implies that the strict transform of $\mathcal{X}$ is locally generated by $\partial_{\widetilde{z}}$. 
\end{proof}
\begin{proof}[Proof of Lemma \ref{lem:claim2}] Since $\boldsymbol{A}_1$ does not have maximal rank but $\boldsymbol{A}$ does, we conclude that $\boldsymbol{\alpha}_1$ is not generated by the rows of $\boldsymbol{A}_1$ (in particular, $\boldsymbol{\alpha}_1\neq 0$). We just need to change the coordinates to $(\pmb{x},\pmb{y},\pmb{w})$ so that, in equation, we get $\pmb{x}^{\boldsymbol{\alpha}_1} = \widetilde{\pmb{x}}^{\boldsymbol{\alpha}_1}(\widetilde{\pmb{y}}-\widetilde{\boldsymbol{\gamma}})^{\boldsymbol{\alpha}_2}$.
\end{proof}
\subsection{Proof of Proposition \ref{prop:Dropping}}
By hypothesis, there exists a local coordinate system $(\pmb{u},v,\pmb{w})$ that satisfies the prepared normal form at $p$ (i.e. it satisfies equation \eqref{eq:basicnormaform} and the conclusions of Proposition \ref{prop:prepnormal}) with $\nu = \nu_p(\theta,\mathcal{I})$. Consider the ideal
\[
\mathcal{J} = (v^\nu, v^j \pmb{u}^{\pmb{r}_{i,j}}b_{i,j})
\]
where we recall that each $b_{i,j}$ is either a unit or zero. Let us consider a sequence of blowings-up:
\[
\tau: (M_r,\theta_r,\mathcal{I}_r,E_r) \rightarrow (M_0,\theta_0,\mathcal{I}_0,E_0)
\]
that principalize $\mathcal{J}$, where $M_0$ is any fixed open neighborhood of $p$ where $\mathcal{J}$ is well-defined. Since $\mathcal{J}$ is generated by monomials in the variables $\pmb{u}$ and $v$, this sequence can be chosen to be combinatorial with respect to the divisor $F:=E \cup \{  v =0 \} $ (see definition \ref{def:combblowingsup}). Furthermore, since $\theta$ is adapted to $E$ and the vector field $\partial_v$ is contained in $\theta$, by remark \ref{rk:combinatorial}, the sequence of blowings-up is $\theta$-admissible.

Now, let $q$ be a point of $M_r$ in the pre-image of $p$. Since $\pmb{\tau}$ is a sequence of combinatorial blowings-up in respect to the divisor $F= E \cup \{v=0\}$, the point $q$ satisfies the hypothesis of either Lemma \ref{lem:claim1} or \ref{lem:claim2}. We divide the analysis in two cases: \medskip

\noindent\emph{Case 1:} We assume we are in conditions of Lemma \ref{lem:claim1} and that equation \eqref{eq:claim1} holds. So, after blowing-up equation \eqref{eq:basicnormaform} become:

\[
\begin{aligned}
  \tau^{\ast}g_1 &= U\pmb{x}^{S_{\nu}}(z-\gamma_{l+1})^\nu+\sum_{j=0}^{\nu-2}\pmb{x}^{S_{1,j}}(z-\gamma_{l+1})^jc_{1,j} \text{ where $U$ is a unity, and}\\
 \tau^{\ast}g_i &= \bar{g}_i \pmb{x}^{S_{\nu}}(z-\gamma_{l+1})^\nu+\sum_{j=0}^{\nu-1}\pmb{x}^{S_{i,j}}(z-\gamma_{l+1})^jc_{i,j}
 \end{aligned}
\]
where the functions $c_{i,j}(\pmb{x},\pmb{y},\pmb{w})$ are either zero or unities, and the unity $U$ can be written as $\widetilde{U}(\pmb{x},\pmb{y},\pmb{w}) + \xmon{\boldsymbol{\alpha}_1} \Omega$ for some function $\Omega$ (note that $\widetilde{U}$ is independent of $z$). We consider two cases depending on which generator of $\mathcal{J}$ pulls back to be a generator of the pull-back of $\mathcal{J}^*$:\medskip

\noindent{\emph{Case 1.1:}}[The pull back of $v^\nu$ generates $\mathcal{J}^{\ast}$, i.e. $S_\nu=\min\{S_\nu,S_{i,j}\}$]
In this case we have:
\[
\begin{aligned}
\tau^{\ast}g_1 = \pmb{x}^{S_{\nu}}[  (\widetilde{U}z+\widetilde{U}\gamma_{l+1} \nu &+ \pmb{x}^{\boldsymbol{\alpha}_1}\Omega_2   ) z^{\nu-1} + \\
& + (\text{terms where the exponent of }z\text{ is }<\nu-1 ]
\end{aligned}
\]
where $\boldsymbol{\alpha}_1$ is a non-zero matrix and $\Omega_2 = [z+\gamma_{l+1} \nu]\Omega$. Since $\widetilde{U}z+\widetilde{U} \nu \gamma_{l+1} + \pmb{x}^{\boldsymbol{\alpha}_1} \Omega_2$ is a unit and the vector field $\partial_{z}$ belongs to $\theta_r$, it is clear that $\nu_q(\theta_r,\mathcal{I}_r) \leq \nu-1<\nu$.\medskip

\noindent{\emph{Case 1.2:}}[There is a maximum $0 \leq j_1 < \nu$ such that the pull back of $\pmb{u}^{r_{i_1,j_1}}v^i$ generates $\mathcal{J}^{\ast}$ for some $i_1$, i.e $S_{i_1,j_1}=\min\{S_\nu,S_{i,j}\}$, $S_\nu> S_{i_1,j_1}$ and $S_{i,j}> S_{i_1,j_1}$ for $j>j_1$]
In this case we have:
\[
\begin{aligned}
\tau^\ast g_{i_1} = \pmb{x}^{S_{i_1,j_1}}\left[(z-\gamma_{l+1})^{j_1}c_{i_1,j_1}+\sum_{j=0}^{j_1-1}\pmb{x}^{S_{i_1,j}-S_{i_1,j_1}}(z-\gamma_{l+1})^{j}c_{i_1,j}(\pmb{x},\pmb{y},\pmb{w}) + \right. \\
\left. \phantom{\sum^{j_1}_{i=1}}+\Omega(\pmb{x},\pmb{y},z,\pmb{w}) \right]
\end{aligned}
\]
where $\Omega( \pmb{0} ,\pmb{y},z,\pmb{w}) \equiv 0$.  Since $c_{i_1,j_1}$ is a unit and the vector field $\partial_{z}$ belongs to $\theta_r$, it is clear that $\nu_q(\theta_r,\mathcal{I}_r) = j_i < \nu$.\medskip

\noindent\emph{Case 2:} We assume we are in conditions of Lemma \ref{lem:claim2} and that equation \eqref{eq:claim2} holds. So, after blowing-up equation \eqref{eq:basicnormaform} become:

\[
\begin{aligned}
  \tau^{\ast}g_1 &= U\pmb{x}^{S_{\nu}}+\sum_{j=0}^{\nu-2}\pmb{x}^{S_{1,j}}c_{1,j} \text{ where $U$ is a unity, and}\\
 \tau^{\ast}g_i &= \bar{g}_i \pmb{x}^{S_{\nu}}+\sum_{j=0}^{\nu-1}\pmb{x}^{S_{i,j}}c_{i,j}
 \end{aligned}
\]
where the functions $c_{i,j}(\pmb{x},\pmb{y},\pmb{w})$ are either zero or unities. We remark that:
\begin{align*}
S_\nu\phantom{j}&= \nu \boldsymbol{\alpha}_1\\
S_{i,j}&= j\boldsymbol{\alpha}_1+\pmb{r}_{i,j}\boldsymbol{A}_1\text{, for }j=0,\ldots,\nu-1.
\end{align*}
So, for a fixed $i$, each $S_\nu$ and $S_{i,j}$ is a sum in the rows of $\boldsymbol{A}_1$ with a different multiple of $\boldsymbol{\alpha}_1$. This means that the exponents $S_\nu$ and $S_{i,j}$ are all distinct because $\boldsymbol{\alpha}_1$ is linearly independent with the rows of $\boldsymbol{A}_1$. Therefore, for each $i$, the generator of $\mathcal{J}^*$ can only be one of the monomials $\pmb{x}^{S_\nu}$ or $\pmb{x}^{S_{i,j}}$. So, let $(i_1,j_1)$ be an index such that $S_{i_1,j_1}=\min\{S_\nu,S_{i,j}\}$ (where we consider, possibly, $(i_1,j_1) = (\nu,1)$). Then:
\[
\tau^\ast g_{i_1} =  \pmb{x}^{S_{i_1,j_1}}U
\]
where $U$ is a unit. In this case, it is clear that $\mathcal{I}_r$ is generated by $\pmb{x}^{S_{i_1,j_1}}$ and that the invariant $\nu$ has dropped to zero.
\section*{Acknowledgments}
I would like to thank Edward Bierstone for the useful suggestions and for reviewing the manuscript. The structure of this manuscript is strongly influenced by him. I would also like to express my gratitude to Daniel Panazzolo for the useful discussions concerning the problem and its applications. Finally, I would like to thank the anonymous reviewer for several very useful comments and, in particular, for suggesting a different title for the manuscript.
\bibliographystyle{alpha}

\begin{thebibliography}{999}
\bibitem{BeloT} A. Belotto da Silva (2013) \textit{Resolution of singularities in foliated spaces}, PhD thesis. Universit\'{e} de Haute-Alsace, France.
\bibitem{Bel1} A. Belotto, \textit{Global resolution of singularities subordinated to a $1$-dimensional foliation}, Journal of Algebra, Volume 447, 2016, Pages 397-423;
\bibitem{Bel2}  A. Belotto da Silva,
\textit{Local monomialization of a system of first integrals}, preprint (2014) arXiv:1411.5333 [math.CV]
\bibitem{AB} A. Belotto da Silva; E. Bierstone, V. Grandjean and P. Milman, \textit{Resolution of singularities of the cotangent sheaf of a singular variety}, preprint (2015) arXiv:1504.07280 [math.AG].

\bibitem{BM} E. Bierstone and P.D. Milman,
\textit{Canonical desingularization in characteristic zero by
blowing up the maximum strata of a local invariant},
Invent. Math. \textbf{128} (1997), 207--302.


\bibitem{BM2} E. Bierstone and P.D. Milman, \textit{Functoriality in resolution of singularities}, Publ. Res. Inst. Math. Sci. 44 (2008), no. 2, 609-639.

\bibitem{Can} F. Cano, \textit{Reduction of the singularities of codimension one singular foliations in dimension three}. Ann. of Math. (2) 160 (2004), no. 3, 907-1011.

\bibitem{Cut2} S. Cutkosky, \textit{Monomialization of morphisms from 3-folds to surfaces}. Lecture Notes in Mathematics, 1786. Springer-Verlag, Berlin, 2002.
\bibitem{Cut1} S. Cutkosky, \textit{Local Monomialization of Analytic Maps}, preprint (2015), arXiv:1504.01299 [math.AG]
\bibitem{Hir} H. Hironaka, \textit{Resolution of singularities of an algebraic variety over a field of characteristic zero}. I, II. Ann. of Math. (2) 79 (1964), 109-203; ibid. (2) 79 1964 205-326.
\bibitem{Hor} L. Hörmander, \textit{An introduction to complex analysis in several variables}, North-Holland Publishing Co. Amsterdam, 1973.
\bibitem{Ko} J. Koll\`{a}r, \textit{Lectures on resolution of singularities}. Annals of Mathematics Studies, 166. Princeton University Press, Princeton, NJ, 2007.
\bibitem{Mc} M. McQuillan, \textit{Canonical models of foliations}, Pure and applied mathematics quarterly, Vol. 4(3), 2008, pp. 877-1012.
\bibitem{McP} M. McQuillan and D. Panazzolo, \textit{Almost \'{E}tale resolution of foliations}. Journal of Differential Geometry 95, no. 2 (2013), 279-319.

\bibitem{Ste} P. Stefan, \textit{Accessibility and foliations with singularities}. Bull. Amer. Math. Soc. 80 (1974), 1142-1145.
\bibitem{Suss} H. Sussmann, \textit{Orbits of families of vector fields and integrability of distributions}. Trans. Amer. Math. Soc. 180 (1973), 171-188.
\bibitem{Pan2} D. Panazzolo, \textit{Resolution of singularities of real-analytic vector fields in dimension three}. Acta Math. 197 (2006), no. 2, 167-289.
\bibitem{Vil4} O. Villamayor, \textit{Constructiveness of Hironaka's resolution}. Ann. Sci. \'{E}cole Norm. Sup. (4) 22 (1989), no. 1, 1-32.
\bibitem{Vil2} O. Villamayor, \textit{Resolution in families}. Math. Ann. 309 (1997), no. 1, 1-19.

\end{thebibliography}

\end{document}